\documentclass[11pt,a4paper]{article}  %
\usepackage{ifdraft}

\usepackage{etoolbox}

\ifdraft{}{\usepackage[numbers,sort&compress]{natbib}}  %
\newcommand{\bibfile}{\jobname.bib}  %
\newcommand{\universalbib}{ref.bib}
\ifdraft{\IfFileExists{\universalbib}{\renewcommand{\bibfile}{\universalbib}}{}}{}
\newcounter{cite}
\pretocmd{\cite}{\stepcounter{cite}}{}{}

\RequirePackage[mathlines]{lineno}
\ifdraft{\linenumbers}{}

\usepackage{amsmath,amsthm,amssymb,amsfonts}
\usepackage{mathtools}  %
\usepackage{empheq}
\usepackage{xcolor}
\usepackage[bbgreekl]{mathbbol}
\DeclareSymbolFontAlphabet{\mathbbm}{bbold}
\DeclareSymbolFontAlphabet{\mathbb}{AMSb}
\usepackage{bbm}
\usepackage{upgreek}
\usepackage{accents}
\usepackage{xspace}
\usepackage{rotating}
\usepackage{multirow,booktabs}
\usepackage[en-US]{datetime2}

\usepackage[normalem]{ulem}
\usepackage[toc,page]{appendix}

\usepackage{sectsty}
\sectionfont{\large}
\subsectionfont{\large}

\definecolor{darkblue}{rgb}{0,0.1,0.5}
\definecolor{darkgreen}{rgb}{0,0.5,0.1}
\definecolor{darkyellow}{rgb}{0.65,0.65,0.01}

\ifdraft{
    \setlength{\marginparwidth}{1.92cm}
    \usepackage[tickmarkheight=3pt,textsize=scriptsize,backgroundcolor=blue!16,linecolor=purple,bordercolor=purple]{todonotes}
}{
    \newcommand{\todo}[1]{}
    
}

\usepackage{graphicx}
\usepackage{tikz,tikzscale,pgf}
\usetikzlibrary{arrows,arrows.meta,patterns,positioning,decorations.markings,shapes}
\usepackage{pgfplots}
\usepackage{pgfplotstable}
\usepackage[justification=centering]{caption}
\usepackage{subcaption}  %
\usepgfplotslibrary{fillbetween}
\pgfplotsset{compat=1.11}

\ifdraft{
    \usepackage{silence}
    \WarningFilter{xcolor}{Incompatible color definition on}
    \WarningFilter{hyperref}{Draft mode on}
    \WarningFilter{refcheck}{Unused label}
    \WarningFilter{microtype}{`draft' option active}
    \WarningFilter{latex}{Writing or overwriting file} %
    \WarningFilter{latex}{Writing file} %
    \WarningFilter{latex}{Tab has} %
    \WarningFilter{latex}{Marginpar on page} %
    \WarningFilter{latex}{author given} %
}{}

\ifdraft{\usepackage{refcheck}
    \newcommand{\url}{\texttt}
}{
    \usepackage{hyperref}
    \hypersetup{colorlinks, linkcolor=darkblue, anchorcolor=darkblue, citecolor=darkblue, urlcolor=darkblue}
    \usepackage{url}
} %
\newcommand{\email}{\texttt}

\usepackage{eqlist}
\usepackage{enumitem}
\setlist[itemize]{leftmargin=*}
\setlist[enumerate]{leftmargin=*,label=\normalfont{(\alph*)}}

\usepackage[section]{algorithm}
\usepackage{algpseudocode,algorithmicx}

\algrenewcommand\algorithmicrequire{\textbf{Input:}}
\algrenewcommand\algorithmicensure{\textbf{Output:}}
\algrenewcommand\alglinenumber[1]{\normalsize #1.}

\newtheorem{theorem}{Theorem}[section]

\newtheorem{corollary}{Corollary}[section]

\newtheorem{lemma}{Lemma}[section]

\newtheorem{proposition}{Proposition}[section]

\newtheorem*{blanketassumption}{Blanket Assumption}
\newtheorem{example}{Example}[section]

\newtheorem{remark}{Remark}[section]
\theoremstyle{definition}
\newtheorem{definition}{Definition}[section]
\usepackage{xpatch}
\xpatchcmd{\proof}{\itshape}{\normalfont\proofnamefont}{}{}
\newcommand{\proofnamefont}{\bfseries}

\numberwithin{equation}{section}

\usepackage{microtype}
\usepackage[nobottomtitles*]{titlesec} %
\mathcode`@="8000{\catcode`\@=\active\gdef@{\mkern1mu}} %

\usepackage{relsize}
\usepackage{nccmath}

\DeclareMathOperator{\gap}{gap}
\DeclareMathOperator{\clarke}{\partial_{\sss{\textnormal{C}\!}}}

\newcommand{\ind}{\mathbbm{1}}

\newcommand{\RR}{\mathbb{R}}
\newcommand{\BB}{\mathcal{B}}

\newcommand{\TT}{\mathcal{T}}

\newcommand{\sol}{\mathcal{S}}

\newcommand{\FF}{\mathcal{F}}

\newcommand{\sss}[1]{{\scriptscriptstyle{#1}}}

\newcommand{\comp}{{\text{c}}}
\newcommand{\trs}{{\scriptstyle{\mathsf{T}}}}

\newcommand{\Mid}{\mathrel{\big|}}
\newcommand{\MID}{\mathrel{\bigg|}}
\newcommand{\mystar}{\textcolor{red}{\boldsymbol{\star}}}
\newcommand{\mycirc}{\textcolor{blue}{\boldsymbol{\circ}}}
\newcommand{\mydot}{\boldsymbol{\cdot}}

\newcommand{\ie}{{i.e.}\xspace}
\newcommand{\eg}{{e.g.}\xspace}
\newcommand{\etal}{{et al.}\xspace}
\newcommand{\iid}{\text{i.i.d.}\xspace}
\newcommand{\as}{\text{a.s.}\xspace}
\newcommand{\io}{\text{i.o.}\xspace}

\newcommand{\me}{\mathrm{e}}

\xspaceaddexceptions{]\}}

\DeclareMathAlphabet{\mathsfit}{T1}{\sfdefault}{\mddefault}{\sldefault}
\SetMathAlphabet{\mathsfit}{bold}{T1}{\sfdefault}{\bfdefault}{\sldefault}

\newcommand{\GG}{\mathcal{G}}
\renewcommand{\Pr}{\mathbb{P}}
\newcommand{\expc}{\mathbb{E}}
\newcommand{\cm}{\operatorname{cm}}
\newcommand{\MD}{\mathcal{D}}
\newcommand{\DD}{\mathfrak{D}}
\newcommand{\ddd}{\mathfrak{d}}

\DeclareFontFamily{U}{dutchcal}{\skewchar\font=45 }
\DeclareFontShape{U}{dutchcal}{m}{n}{<-> s*[1.0] dutchcal-r}{}
\DeclareFontShape{U}{dutchcal}{b}{n}{<-> s*[1.0] dutchcal-b}{}
\DeclareMathAlphabet{\mathlcal}{U}{dutchcal}{m}{n}
\SetMathAlphabet{\mathlcal}{bold}{U}{dutchcal}{b}{n}

\usepackage[scr=boondoxupr]{mathalfa}

\DeclareMathAlphabet{\mathpzc}{OT1}{pzc}{m}{it} %

\title{Non-convergence Analysis of Probabilistic Direct Search
}

\date{\today}

\author{
Cunxin Huang\thanks{
Department of Applied Mathematics, The Hong Kong Polytechnic University,
Hong Kong, China. Email:~\email{cun-xin.huang@connect.polyu.hk}.
}\and Zaikun Zhang\thanks{
School of Mathematics, Sun Yat-sen University,
Guangzhou, China. Email:~\email{zhangzaikun@mail.sysu.edu.cn}.
}
}

\begin{document}

\maketitle

\begin{abstract}
We present a non-convergence theory for probabilistic direct search,
a randomized derivative-free optimization method, where non-convergence means
the failure to produce iterates that achieve stationarity asymptotically.
The motivation is to
understand whether the submartingale-like assumption in the existing convergence theory is essential or
merely an artifact of the analysis techniques.
For convex objectives, we prove that
the probability of non-convergence is positive,
provided that the polling directions satisfy a probabilistic ascent condition that is essentially the opposite
of the submartingale-like convergence condition. Furthermore, we establish a lower bound for the non-convergence probability.
For the typical implementation of this method, where each iteration draws a fixed number of random
polling directions independently and uniformly from the unit sphere,
our theory implies that the method is not globally convergent
if the number of directions is below the threshold specified in the convergence theory,
and the submartingale-like assumption is confirmed to be essential for convergence.
Our theory is obtained by examining two random series that control the distance from any iterate
to the starting point and estimating the probability for these series to stay below certain bounds.
\end{abstract}


\section{Introduction}
\label{sec:introduction}

When does your algorithm fail to converge? This question is arguably as fundamental as asking when it
converges. A theoretical investigation of this issue can potentially deepen our understanding of the
algorithm, inform its practical implementation, and, in particular, provide a framework for
assessing whether the assumptions in existing convergence theory are genuine necessities or merely
technical artifacts. However, this question has received far less attention than its convergence
counterpart.

We address this issue for a randomized derivative-free optimization~(DFO) algorithm
known as probabilistic direct search~(PDS)~\cite{Gratton_Royer_Vicente_Zhang_2015,
Gratton_Royer_Vicente_Zhang_2019}.
Our answer is provided in the form of a \emph{non-convergence analysis} of~PDS.
We will focus on the unconstrained optimization problem
\begin{equation}\label{eq:uncopt}
    \min_{x\in \RR^n} f(x),
\end{equation}
where~$f\mathrel{:} \RR^n\to\RR$~is the objective function.

Direct search~\cite{Kolda_Lewis_Torczon_2003,Dzahini_Rinaldi_Royer_Zeffiro_2025}
is a class of DFO methods that define iterates based on comparisons of function values
sampled following a certain scheme without building models for the objective.
We focus on directional direct search methods
requiring sufficient decrease~\cite[Section~7.7]{Conn_Scheinberg_Vicente_2009b}.
The deterministic variant of these methods evaluates at least~$n+1$
function values per iteration in the worst case, which is prohibitively expensive even for modestly big~$n$.

To alleviate the computational burden per iteration, Gratton \etal~\cite{Gratton_Royer_Vicente_Zhang_2015}
propose PDS, which searches along a set of directions randomly generated at each iteration.
In the typical implementation of PDS, each iteration draws~$m$~\iid~random polling directions
uniformly from the unit sphere
with no dependence on existing polling directions or iterates.
PDS is shown to converge globally if
the random directions form a sequence of~$p_0$-probabilistic~$\kappa$-descent
sets~\mbox{\cite[Definition~3.1]{Gratton_Royer_Vicente_Zhang_2015}}
with~$\kappa>0$ and
\begin{equation}
  \label{eq:p0}
  p_0 \;=\; \frac{\log \theta}{\log (\gamma^{-1}\theta)},
\end{equation}
where~$\theta \in (0, 1)$ and~$\gamma \ge 1$ are parameters that the algorithm uses to update the step size.
In particular, the above-mentioned typical implementation of PDS converges
globally if~$\gamma > 1$ and
\begin{equation}
    \label{eq:m_con}
    m \;>\; \log_2 \left( 1 - \frac{\log \theta}{\log \gamma} \right),
\end{equation}
where the right-hand side is independent of problem~\eqref{eq:uncopt}, especially its
dimension~$n$.

A natural question then arises: what if~\eqref{eq:m_con} is not satisfied?
More fundamentally, is the~$p_0$-probabilistic~$\kappa$-descent assumption essential for the
convergence of PDS or is it only a technical artifact in the existing convergence analysis?
Establishing a theory (rather than an example) to answer these questions is the main goal of our non-convergence analysis.

However, the motivation for our investigation is not limited to PDS. The probabilistic
descent assumption is a representative of the \emph{submartingale-like assumptions}
first proposed in~\cite{Bandeira_Scheinberg_Vicente_2014} and now widely used in
the convergence theory of randomized optimization algorithms, including probabilistic
trust region~\cite{Bandeira_Scheinberg_Vicente_2014,Gratton_Royer_Vicente_Zhang_2015,Wang_Yuan_2022},
line search~\cite{Cartis_Scheinberg_2018,Berahas_Cao_Scheinberg_2021}, cubic
regularization~\cite{Cartis_Scheinberg_2018}, and subspace methods~\cite{Cartis_Roberts_2022,
Roberts_Royer_2023}.
We hope that our work will not only shed light on the necessities of
such assumptions, but also
provide inspiration and tools for the non-convergence analysis of other randomized methods,
thus contributing to a deeper understanding of these methods.

The main discoveries of our non-convergence analysis are as follows.
If the polling direction sets of PDS form a sequence of~$p$-probabilistic ascent
sets~(Definition~\ref{def:prob-ascent}) with~$p > p_*$, where
\begin{equation}\label{eq:define-p_*}
    p_* \;=\; 1-p_0 \;=\; \frac{\log \gamma}{\log (\theta^{-1}\gamma)},
\end{equation}
then PDS does not converge globally on convex objectives~(Theorem~\ref{thm:qualitative}).
For the aforementioned typical implementation of PDS, the non-convergence condition reduces to
\begin{equation}
    \nonumber
    m \;<\; \log_2 \left(1 - \frac{\log \theta}{\log \gamma}\right),
\end{equation}
a requirement opposite to the convergence condition~\eqref{eq:m_con} except for the boundary situation
with $m = \log_2(1-\log\theta / \log \gamma)$, which cannot be covered by the existing convergence
analysis~\cite{Gratton_Royer_Vicente_Zhang_2015} or our theory.
As a highlight, we not only prove that PDS fails to converge
globally under the above-mentioned conditions, but also establish a lower bound on the non-convergence
probability~(Theorem~\ref{thm:quantitative}), which appears to be sharp in numerical experiments.

We must stress that the word ``non-convergence'' in our paper does not mean divergence.
Here, non-convergence signifies that the iterates do not achieve any kind of stationarity asymptotically,
but they may converge to a non-stationary point~(see Remark~\ref{rem:non_conv}).
Indeed, non-convergence investigations are not uncommon in optimization research.
For instance,
BFGS~\cite{Dai_2013},
ADMM~\cite{Chen_He_Ye_Yuan_2016},
and Adam~\cite{Zhang_Chen_Shi_Sun_Luo_2022} are all shown to be non-convergent
under certain conditions.
Particularly,
Audet~\cite{Audet_2004} exhibits non-convergence examples of
Generalized Pattern Search~(GPS), a classic direct search method, thereby
confirming that the conditions imposed in its convergence theory are all indispensable.
Other papers on non-convergence include~\cite{Powell_1973,Thompson_1977,Yuan_1998,BenGharbia_Gilbert_2012,
Mascarenhas_2014,Ramponi_2018,Hong_Zeng_Zhang_Sun_2022}.
Most of these works, however, focus on constructing \emph{non-convergence examples}, whereas our paper aims
to establish a systematic \emph{non-convergence theory}.
Examples are valuable, but isolated examples, especially pathological ones, may fail to reflect
typical behavior. A theory, by contrast, can often provide a broader picture and reveal deeper
mechanisms.

The remaining sections are organized as follows. In Section~\ref{sec:preliminary}, we
provide a concise review of DFO and introduce preliminary concepts about PDS.
Section~\ref{sec:nonconvergence} contains the major results of this paper.
It establishes the non-convergence theory and shows that
the typical implementation of PDS is not globally convergent if the number of polling directions is less
than~$\log_2 (1 - \log \theta / \log \gamma)$, with the non-convergence probability quantified.
We extend our theory to the nonsmooth case in~Section~\ref{sec:nonsmooth} and
conclude with some perspectives in~Section~\ref{sec:conclusion}. The appendices contain some lemmas,
proofs, and discussions, all of which can be skipped without affecting the understanding of the main
theory.

\textbf{Notations.}~For an event~$E$, we use~$\ind(E)$ to denote the random variable such that
\begin{equation*}
    \ind(E) \;=\;
    \begin{cases}
        1, & \text{if $E$ happens},\\
        0, & \text{otherwise}.
    \end{cases}
\end{equation*}
The abbreviation~``\as'' stands for ``almost surely'' and~``\io'' for~``infinitely often''.
The Euclidean norm is denoted by~$\|\cdot\|$,
and~$\BB(x, r)$ represents the open Euclidean ball centered at~$x\in\RR^n$ with radius~$r> 0$.
For the objective function~$f$ of problem~\eqref{eq:uncopt}, we denote
\begin{align*}
    \inf f &\;=\; \inf_{x\in\RR^n} f(x), \\
    \sol(f) &\;=\; \{x\in\RR^n\mathrel{:} f(x) = \inf f\}.
\end{align*}
Note that~$\inf f$ may be~$-\infty$ and~$\sol(f)$ may be empty.
As in~\cite[page~113]{Rockafellar_Wets_1998}, we define the gap distance between two
sets~$A, B \subseteq \RR^n$ as
\begin{equation*}
    \gap(A, B) = \inf \{\|a - b\|\mathrel{:} a\in A, b\in B \},
\end{equation*}
which is supposed to be~$\infty$ if~$A = \emptyset$ or~$B = \emptyset$; if~$A$ is a singleton~$\{a\}$, then we write~$\gap(a, B)$ instead of~$\gap(\{a\}, B)$.
As a convention, we define the summation and product over an empty index set as~$0$ and~$1$,
respectively. In particular, $\sum_{k=i}^j a_k= 0$ and~$\prod_{k=i}^j a_k=1$ for any real
sequence~$\{a_k\}$ whenever~$i > j$.
We write~$\lim_k$ instead of~$\lim_{k\to\infty}$ in inline equations for brevity, which
applies to lower and upper limits as well.

\section{Preliminaries}\label{sec:preliminary}

Within the existing literature, DFO methods are broadly classified into two primary
categories: model-based methods and direct search methods~\cite{Conn_Scheinberg_Vicente_2009b,
Audet_Hare_2017}.
Model-based methods construct local models of the problem based on function values
and exploit such models under a trust region~\cite{Conn_Scheinberg_Vicente_2009a} or
line search~\cite{Berahas_Byrd_Nocedal_2019} framework.
A wealth of classical literature on model-based methods can be found in
\cite{Powell_1994,Powell_2002a,Powell_2006,Powell_2009,Conn_Scheinberg_Vicente_2009a,Berahas_Byrd_Nocedal_2019},
to name but a few.
Direct search methods do not explicitly build models for the problem,
but instead, they define iterates based on comparisons of function values sampled following
a certain scheme~\cite{Kolda_Lewis_Torczon_2003,Dzahini_Rinaldi_Royer_Zeffiro_2025}.
There exist multiple types of direct search methods, examples
including the Nelder--Mead simplex method~\cite{Nelder_Mead_1965},
GPS~\cite{Torczon_1997,Lewis_Torczon_1999, Audet_Dennis_2002},
the MADS family~\cite{Audet_Dennis_2006,LeDigabel_2011,Audet_LeDigabel_Montplaisir_Tribes_2022},
and BFO~\cite{Porcelli_Toint_2017,Porcelli_Toint_2020b}.
More extensive surveys on DFO can be found
in the monographs~\cite{Conn_Scheinberg_Vicente_2009b,Audet_Hare_2017},
the review papers~\cite{Rios_Sahinidis_2013,Custodio_Scheinberg_Vicente_2017,Larson_Menickelly_Wild_2019},
and the references therein.

Probabilistic techniques have been introduced to both categories of methods in the past
decade~\cite{Bandeira_Scheinberg_Vicente_2014, Gratton_Royer_Vicente_Zhang_2015,
Cartis_Scheinberg_2018,Gratton_Royer_Vicente_Zhang_2018,Gratton_Royer_Vicente_Zhang_2019,
Cartis_Roberts_2022,Roberts_Royer_2023},
PDS being a result in the direct search category.
In what follows, we first present a basic framework of direct search
adopted from~\cite[Section~7.7]{Conn_Scheinberg_Vicente_2009b} and~\cite{Vicente_2013},
and then introduce the probabilistic variant of this framework proposed
in~\cite{Gratton_Royer_Vicente_Zhang_2015} for PDS, around which our investigation will revolve.

\subsection{Direct search based on sufficient decrease}\label{ssec:ds_sufficient_decrease}

Algorithm~\ref{alg:direct_search} presents a direct search method for solving
problem~\eqref{eq:uncopt}. Inequality~\eqref{eq:sufficient_decrease} is the sufficient decrease condition,
where the forcing function~$\rho\mathrel{:}(0,\infty)\to(0,\infty)$~is nondecreasing
and~$\rho (\alpha) = o(\alpha)$~when~$\alpha\to 0^+$, a typical choice being~$\rho (\alpha) = c\alpha^2/2$ with a positive constant~$c$.

\begin{algorithm}[htbp!]
  \caption{\label{alg:direct_search}Deterministic direct search based on sufficient decrease}
  Select~$x_0\in \RR^n$,~$\alpha_0>0$,~$\theta\in (0,1)$,~$\gamma\in [1,\infty)$, and a forcing function~$\rho$.
  \\For~$k = 0, 1, 2, \dots$, do the following.
  \begin{algorithmic}[1]
    \State Generate a set of directions~$\MD_k\subseteq \RR^{n}$ deterministically.
    \State If there exists a direction~$d_k\in \MD_k$ such that
    \begin{equation}\label{eq:sufficient_decrease}
        f(x_k) - f(x_k + \alpha_k d_k) \;>\; \rho(\alpha_k),
    \end{equation}
    then set~$x_{k+1} = x_k + \alpha_k d_k$ and~$\alpha_{k+1} = \gamma \alpha_k$; otherwise,
    set~$x_{k+1} = x_k$ and~$\alpha_{k+1} = \theta \alpha_k$.
  \end{algorithmic}
\end{algorithm}

Step~2 of Algorithm~\ref{alg:direct_search} is known as
\emph{polling}~\cite[Chapter~7]{Conn_Scheinberg_Vicente_2009b}, and the directions
in~$\MD_k$ are called the \emph{polling directions}.
In practice, a search step may be taken at the beginning of each
iteration~(see~\cite[Algorithm~3.2]{Kolda_Lewis_Torczon_2003}).
As in~\cite{Gratton_Royer_Vicente_Zhang_2015}, we omit such an option and focus on polling.

Algorithm~\ref{alg:direct_search} represents only one of the many types of direct search.
Instead of the sufficient decrease condition~\eqref{eq:sufficient_decrease},
there are frameworks that adopt the simple decrease condition
\begin{equation}
 \label{eq:simple_decrease}
    f(x_k) - f(x_k + \alpha_k d_k) \;>\; 0
\end{equation}
and guarantee global convergence by integrality and rationality requirements on the polling directions and
the step sizes~\cite{Torczon_1997,Audet_Dennis_2002,Audet_2004}~\textnormal{(}see
also~\cite[Section~7.6]{Conn_Scheinberg_Vicente_2009b}\textnormal{)}.

To implement Algorithm~\ref{alg:direct_search}, a polling strategy is needed to select the
direction~$d_k$ if there are multiple candidates satisfying~\eqref{eq:sufficient_decrease}.
Two common strategies exist: complete polling chooses the direction that decreases the function value
the most, picking the first in case of a tie; opportunistic polling takes the first direction
fulfilling~\eqref{eq:sufficient_decrease}.
We also need to set an order for evaluating~$\{f(x_k+\alpha_k d)\mathrel{:} d\in\MD_k\}$.
A strategy suggested in~\cite[Section~4]{Custodio_Vicente_2007} decides the order by an oracle that
can help us rank the prospective decreases of~$f$ along the polling directions, the oracle
in~\cite{Custodio_Vicente_2007} being a direction of potential descent.
For generality, Algorithm~\ref{alg:direct_search} deliberately keeps the strategies of polling and ordering unspecified.

The analysis of Algorithm~\ref{alg:direct_search} depends on the concept of the cosine measure defined below.
\begin{definition}[Cosine measure]
    \label{def:cm}
    Let~$\MD$~be a finite and nonempty set of nonzero vectors in~$\RR^n$. The cosine measure of~$\MD$
    with respect to a nonzero vector~$v$, denoted by~$\cm(\MD,v)$, is defined~as
    \begin{equation*}
      \cm(\MD,v) \;=\; \max_{d \in \MD} \frac{d^{\trs} v}{\|d\| \|v\|}.
    \end{equation*}
    The cosine measure of~$\MD$, denoted by~$\cm(\MD)$, is defined
    as~$\cm(\MD) = \min _{v \in \mathbb{R}^n \backslash\{0\}} \cm(\MD,v)$.
\end{definition}
\begin{remark}\label{rem:def_cm}
    Definition~\ref{def:cm} does not specify the value of~$\cm(\,\cdot\,, 0)$.
    As a convention, we suppose that it is defined to be a constant in~$[-1,1]$.
    For example,~$\cm(\,\cdot\,, 0) = 1$ in~\cite{Gratton_Royer_Vicente_Zhang_2015}.
    We choose not to particularize this constant, because its value will not affect our non-convergence analysis.
    See Remark~\ref{rem:cm_not_affect} for more details.
\end{remark}

If~$f$ is smooth and there exists a constant~$\kappa>0$~such that~$\cm(\MD_k) \ge \kappa$~for each~$k\ge 0$, then Algorithm~\ref{alg:direct_search} converges under some technical assumptions. See~\cite[Theorem~3.11]{Kolda_Lewis_Torczon_2003}.

\subsection{Probabilistic direct search}\label{ssec:prob_ds}

Algorithm~\ref{alg:direct_search_r} presents the PDS method
proposed in~\cite{Gratton_Royer_Vicente_Zhang_2015}. It is the same as
Algorithm~\ref{alg:direct_search} except that the polling directions in Step~1 are random vectors
over a probability space~$(\Omega, \FF, \Pr)$. Consequently, the iterates and the step sizes are
also random, although the starting point and the initial step size are still chosen deterministically.

\begin{algorithm}[htbp!]
    \caption{\label{alg:direct_search_r}Probabilistic direct search (PDS)}
  \vspace{0.4em}
  Identical to Algorithm~\ref{alg:direct_search} except that the polling directions in Step~1 are
  generated randomly.%
  \vspace{0.4em}
\end{algorithm}

For a clear discussion of Algorithm~\ref{alg:direct_search_r}, it is necessary to use different
notations for random elements and their realizations.
Similar to~\cite{Gratton_Royer_Vicente_Zhang_2015}, we adopt the notations summarized in Table~\ref{tab:notation}.
\begin{table}[htbp!]
  \centering
  \caption{Notations for random elements and their realizations at iteration~$k$}
  \label{tab:notation}
  \begin{tabular}{ccccc}
      \toprule
       &Polling direction set &Iterate      &Step size  & Gradient at the iterate \\
      \toprule
      Random element &$\DD_k$ &$X_k$      &$A_k$  & $G_k$\\
      Realization     &$\MD_k$ &$x_k$ &$\alpha_k$  & $g_k$\\
      \bottomrule
  \end{tabular}
\end{table}

We make the following blanket assumption on the sequence of polling direction sets~$\{\DD_k\}$ to
simplify our presentation, although our analysis remains valid after slight modifications if the
lengths of the polling directions are only uniformly bounded.
\begin{blanketassumption}\label{as:length}
For each~$k\ge 0$, the set~$\DD_k$ is nonempty and consists of finitely many unit random vectors.
\end{blanketassumption}

The study of Algorithm~\ref{alg:direct_search_r} heavily relies on the concept of~$\sigma$-algebras and conditional probability with respect to them~\cite[Section~4.1]{Durrett_2019}.
For each~$k\ge 0$, we define
\begin{equation}
\label{eq:Fk}
\FF_k \;=\; \sigma (\DD_0,\,X_1,\, \dots,\,\DD_k,\,X_{k+1}),
\end{equation}
which is the~$\sigma$-algebra generated by~$\DD_0$, $X_1$, \dots, $\DD_k$, $X_{k+1}$.
Note that~$\FF_k$ does not involve~$X_0$, which is deterministic.
In addition, we define
\[
    \FF_{-1} \;=\; \{\emptyset,\,\Omega\}.
\]
Roughly speaking, for~$k\ge 0$,~$\FF_k$ captures the information about the polling directions and iterates up to
the end of iteration~$k$, when~$X_{k+1}$ has been generated but $\DD_{k+1}$ has not.
Clearly,~$\DD_k$ is~{$\FF_k$-measurable} and~$X_k$ is~\mbox{$\FF_{k-1}$-measurable}. %
In addition,~$G_k$ is~\mbox{$\FF_{k-1}$-measurable} if~$f$ is continuously differentiable,
and~$A_k$ is~$\FF_{k-1}$-measurable by mathematical induction based on the recurrence
\begin{equation}
    \label{eq:Ak_recur}
    A_{k+1} \;=\; \gamma^{\ind(X_{k+1} \ne X_k)} \theta^{\ind(X_{k+1} = X_k)} A_k,
\end{equation}
which holds because~$\{A_{k+1} = \gamma A_k\} = \{X_{k+1} \ne X_k\}$ and~$\{A_{k+1} = \theta A_k\}
= \{X_{k+1} = X_k\}$.
In the language of probability theory, $\{\DD_k\}$ is adapted to~$\{\FF_k\}$,
while~$\{X_k\}$,~$\{G_k\}$, and~$\{A_k\}$ are predictable with respect
to~$\{\FF_k\}$~(see~\cite[Section~4.1]{Durrett_2019}).

\subsection{Global convergence of PDS}\label{ssec:conv_pds}

The global convergence theory~\cite{Gratton_Royer_Vicente_Zhang_2015} of Algorithm~\ref{alg:direct_search_r} is summarized below for later
reference.

\begin{definition}[{\cite[Definition 3.1]{Gratton_Royer_Vicente_Zhang_2015}}]\label{def:p-k_descent}
    Let~$p \in [0,1]$ and $\kappa \in [-1,1]$.
    Consider Algorithm~\ref{alg:direct_search_r} with~$f$ being continuously differentiable
    on~$\RR^n$. The sequence~$\{\DD_k\}$ is said to be a sequence of~$p$-probabilistic~$\kappa$-descent sets if
    \begin{equation}\label{eq:p-k_descent-def}
      \Pr\left(\cm\left(\DD_k,-G_k\right) \ge \kappa \mid \FF_{k-1}\right) \;\ge\; p \quad \text{for each~$k\ge 0$}.
    \end{equation}
\end{definition}

\begin{theorem}[{\cite[Theorem~3.4]{Gratton_Royer_Vicente_Zhang_2015}}]\label{thm:conv_pds}
  Consider Algorithm~\ref{alg:direct_search_r} with~$f$ being continuously differentiable and
  bounded below on~$\RR^n$, and~$\nabla f$ being Lipschitz continuous on~$\RR^n$.
  If~$\{\DD_k\}$~is a sequence
  of~$p_0$-probabilistic~$\kappa$-descent sets with~$p_0$ being defined in~\eqref{eq:p0}
  and~$\kappa$ being a positive constant,
  then~$\Pr(\liminf_{k}\|G_k\|=0) = 1$.
\end{theorem}

\begin{remark}
    \label{rem:as}
    The probability in~\eqref{eq:p-k_descent-def} is a conditional probability with respect to
    a~$\sigma$-algebra.
    It is a random variable and is only defined up to almost sure
    equivalence~\textnormal{(}\eg,~\cite[Remark~1 on page~213]{Shiryaev_1996}
    and~\cite[Theorem~8.12]{Klenke_2020}\textnormal{)}.
    Consequently, following the convention in probability
    theory~\textnormal{(}\eg,~\cite[page~179]{Durrett_2019} and~\cite[page~195]{Klenke_2020}\textnormal{)},
    the inequality in Definition~\ref{def:p-k_descent} should be understood in the almost sure sense.
    Henceforth, all the equalities and inequalities
    should be interpreted in this way if they involve conditional probabilities or expectations with
    respect to a~$\sigma$-algebra \textnormal{(}\eg, condition~\eqref{eq:prob-asc-def} in
    Definition~\ref{def:prob-ascent}\textnormal{)}, and we will not repeat this point every time.
\end{remark}

\begin{remark}
The~$\sigma$-algebra~$\FF_{k}$ defined in~\eqref{eq:Fk}
reduces to~$\FF_{k}^{\DD}=\sigma(\DD_0, \dots, \DD_{k})$
if we assume that~$X_{\ell}$ is measurable with respect to~$\FF_{\ell-1}^{\DD}$ for each~$\ell\ge 1$.
As will be clarified in Lemma~\ref{lem:measurable-polling},
this assumption is fulfilled by the implementations of Algorithm~\ref{alg:direct_search_r} considered
in~\cite{Gratton_Royer_Vicente_Zhang_2015},
but it may fail if we allow the unspecified polling strategy in
the algorithm to involve randomness beyond the polling directions
\textnormal{(}Example~\ref{exp:measurable-polling}\textnormal{)}. Hence, we choose not to
impose such an assumption.
In this sense, Theorem~\ref{thm:conv_pds} is a slight generalization
of~\cite[Theorem~3.4]{Gratton_Royer_Vicente_Zhang_2015}, although the proof is essentially the same.%
\end{remark}

Corollary~\ref{cor:conv_pds} is the specialization of Theorem~\ref{thm:conv_pds} for the typical
implementation of PDS mentioned in Section~\ref{sec:introduction}.

\begin{corollary}[{\cite[Corollary B.4]{Gratton_Royer_Vicente_Zhang_2015}}]
    \label{cor:conv_pds}
    Consider Algorithm~\ref{alg:direct_search_r} with~$f$ satisfying the assumptions in
    Theorem~\ref{thm:conv_pds}.
    Let each~$\DD_k$ be a set of~$m$~\iid~random
    vectors uniformly distributed on the unit sphere in~$\RR^n$
    with no dependence on existing polling directions or iterates.
    If~$\gamma >1$ and~$m > \log_2(1-{\log \theta}/{\log \gamma})$,
    then~$\Pr(\liminf_{k}\|G_k\|=0) = 1$.
\end{corollary}

The global convergence rate of PDS can also be established under a probabilistic descent
assumption~\cite[Section~4]{Gratton_Royer_Vicente_Zhang_2015}, but it is not the focus of this paper.

\section{Probabilistic ascent and non-convergence analysis of PDS}\label{sec:nonconvergence}

How will Algorithm~\ref{alg:direct_search_r} behave if the polling direction sets~$\{\DD_k\}$
fail to satisfy the probabilistic descent condition in Theorem~\ref{thm:conv_pds}?
We now address this question by introducing the concept of probabilistic
ascent and developing the non-convergence theory of Algorithm~\ref{alg:direct_search_r} based
on it.

Our discussion begins with two motivating examples to build intuition. We then define probabilistic
ascent and prove a key lemma (Lemma~\ref{lem:X_k-to-U_k})
that establishes the framework for the subsequent analysis.
This framework yields two non-convergence
results for Algorithm~\ref{alg:direct_search_r}:
a basic one using Markov's inequality (Theorem~\ref{thm:loose}) and a refined characterization via
a conditional Chernoff bound~(Theorems~\ref{thm:qualitative} and~\ref{thm:quantitative}).
The latter constitutes the main technical contribution of the section and enables us to
show that the probabilistic descent assumption in Theorem~\ref{thm:conv_pds} is necessary for the
global convergence of the algorithm under some conditions~(Theorem~\ref{thm:equiv}).
We then revisit our motivating examples and close by proposing a condition weaker than
probabilistic ascent to broaden our theory.

\subsection{Motivating examples}\label{ssec:examples}

\subsubsection{Failure of global convergence: a numerical illustration}
\label{sssec:numerical_illust}

Before diving into the analysis, we conduct a simple test to illustrate the behavior of
Algorithm~\ref{alg:direct_search_r} when
the probabilistic descent condition in the convergence theory is not satisfied. We will focus
on the typical implementation of the algorithm specified in Corollary~\ref{cor:conv_pds}.

As a simple illustration, let the objective function be~$f(x) = x^\trs x$ with~$x\in\RR^2$. We set the
forcing function~$\rho(\alpha) = 10^{-3}\alpha^2$,
the starting point~$x_0 = (-10, 0)^\trs$, the initial step size~$\alpha_0=1$, the shrinking
factor~$\theta = 1/4$, and the expanding factor~$\gamma = 3/2$.
At each iteration, the polling direction set consists of~$m = 2$~random vectors
that are~\iid~and uniformly distributed on the unit circle with no dependence on existing polling directions or iterates.
Note that~$\log_2(1 - \log \theta / \log \gamma) \approx 2.14 > m$,
violating the condition in Corollary~\ref{cor:conv_pds} for convergence.
The polling strategy is complete polling.
The algorithm is terminated when the step size drops below the machine
epsilon~($\approx 2 \times 10^{-16}$) or the number of iterations reaches~$10^3$.

We run the algorithm for~$10^4$ times independently.
The results are shown in Figure~\ref{fig:nonconvergence}, where the circle represents the
starting point~$x_0$, the pentagram represents the global minimizer~$(0,0)^\trs$, and each dot represents the best
iterate~(\ie, the one with the lowest function value) obtained in a run of the algorithm.

\begin{figure}[htbp!]
  \centering
  \includegraphics[width=0.6\textwidth]{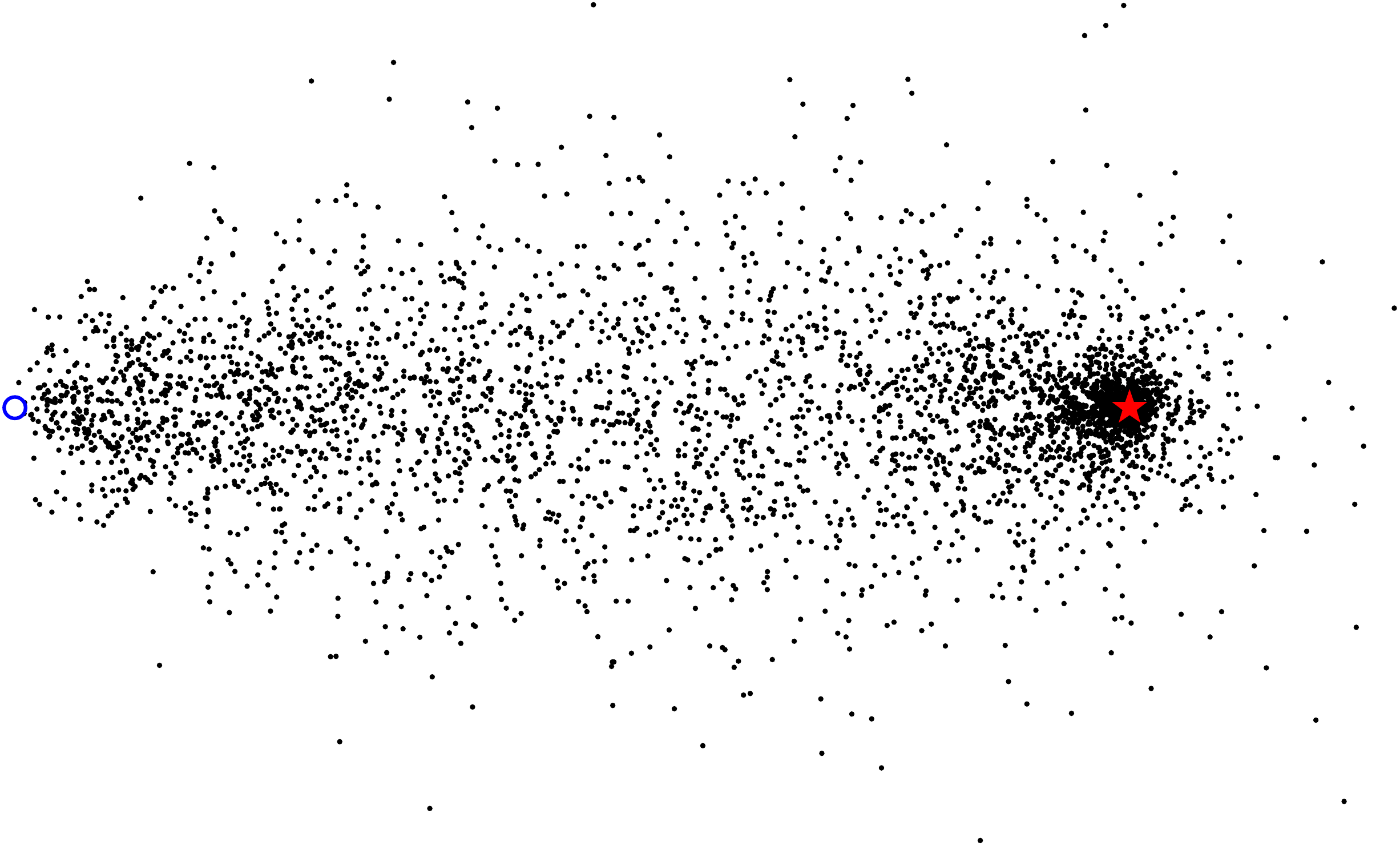}
  \begin{center}
      $\mycirc$~Starting point \qquad
      ${\mystar}$~Global minimizer \qquad
      $\mydot$~Best iterate of a run of the algorithm
  \end{center}
  \caption{A test illustrating failure of convergence of Algorithm~\ref{alg:direct_search_r}}
  \label{fig:nonconvergence}
\end{figure}

As we can observe, many of these dots are far away from the global minimizer.
Even though we cannot draw any rigorous conclusion about the asymptotic behavior
of Algorithm~\ref{alg:direct_search_r} based on this finite-time observation, it motivates us to conjecture that
the algorithm is not globally convergent under this setting. We will confirm this conjecture in
the subsequent analysis, and furthermore, estimate the probability that the iterates of
the algorithm stay away from the minimizer~(see~Corollary~\ref{cor:nonconv_pds}).
The estimation will be verified numerically by a refined version of the above experiment in
Subsection~\ref{sssec:empirical_rate}.

\subsubsection{A tiny example reflecting the big picture}
\label{ssec:small_example}

As a preview, Example~\ref{exp:simple_illustrate} presents a particularly
simple instance that illustrates our non-convergence theory,
highlighting the sharp dichotomy between the convergence and non-convergence conditions.
We will provide detailed explanations about this example
in Subsection~\ref{sssec:explain_example},
although the convergence part follows from~\cite{Gratton_Royer_Vicente_Zhang_2015}~(see
also Corollary~\ref{cor:conv_pds}).

\begin{example}\label{exp:simple_illustrate}
Let~$n = 1$, $f(x) = x^2$, and~$x_0$ be a nonzero number. Consider Algorithm~\ref{alg:direct_search_r}
with~$\DD_k = \{\ddd_k\}$,
where~$\ddd_k$~is a random variable independent of~$\FF_{k-1}$ and takes either~$1$ or~$-1$, each with probability~$1/2$.
Then~$\Pr(X_k \to 0) = 1$ {if and only if}~$\theta \gamma \ge 1$.
When~$\theta\gamma < 1$, there exist constants~$C >0$ and~$\bar{\zeta} >0$ such that
\begin{equation}
    \label{eq:simple_illustrate}
    \Pr\left( \inf_{k\ge 0}|X_k| \,\ge\, (1-\zeta)|x_0| \right)
    \;\ge\;
    C \zeta^{-\frac{\log 2}{\log \theta}}
    \quad\text{for}\;~\zeta \in (0, \overline{\zeta}).
\end{equation}
\end{example}

\subsection{Probabilistic ascent}\label{ssec:prob_ascent}

The concept of probabilistic ascent defined below will play a central role in our non-convergence
analysis, mirroring the role of probabilistic descent in the convergence theory.
\begin{definition}[$p$-probabilistic ascent]\label{def:prob-ascent}
    Let $p \in [0,1]$. Consider Algorithm~\ref{alg:direct_search_r} with~$f$ being continuously
    differentiable on~$\RR^n$. The sequence~$\{\DD_k\}$~is said to be a sequence
    of~$p$-probabilistic ascent sets if%
    \begin{equation}\label{eq:prob-asc-def}
      \Pr\left(\cm(\DD_k,-G_k) \le 0 \mid \FF_{k-1} \right) \;\ge\; p \ind\left(G_k \ne 0\right) \quad \text{for each~$k\ge 0$}.
    \end{equation}
\end{definition}

The event~$\{\cm(\DD_k,-G_k) \le 0\}$ means that $\DD_k$ does not contain any descent direction.
Condition~\eqref{eq:prob-asc-def} requires that the probability of this event given the past
is at least~$p$ whenever~$G_k \ne 0$.
One may ask whether we can omit the indicator~$\ind(G_k \ne 0)$ from
condition~\eqref{eq:prob-asc-def}.
Readers interested in this subtlety can refer to Appendix~\ref{app:disc_def}.

Proposition~\ref{prop:prob-asc} shows that the polling direction sets used by the typical
implementation of Algorithm~\ref{alg:direct_search_r} specified in Corollary~\ref{cor:conv_pds}
are probabilistic ascent.
The proof is given in Appendix~\ref{app:proofs}.

\begin{proposition}\label{prop:prob-asc}
  Consider Algorithm~\ref{alg:direct_search_r} with~$f$ being continuously differentiable on~$\RR^n$.
  The sequence~$\{\DD_k\}$ specified in Corollary~\ref{cor:conv_pds} is a sequence
  of~$p$-probabilistic ascent sets with~$p = 2^{-m}$.%
\end{proposition}

Our analysis will heavily depend on the~0-1 process~$\{Y_k\}$ with
\begin{equation}\label{eq:define_Y_k}
    Y_k = \ind\Big(\min_{\ddd\in\DD_k} \ddd^\trs G_k < 0 \Big)
    \quad \text{for each } k\ge 0,
\end{equation}
which provides us with an equivalent definition of probabilistic ascent as stated in Lemma~\ref{lem:submartingale-Y_k}. This equivalence is
a simple consequence of Proposition~\ref{prop:prob-asc-def-equiv}.
\begin{lemma}\label{lem:submartingale-Y_k}
  Consider Algorithm~\ref{alg:direct_search_r} with~$f$ being continuously differentiable on~$\RR^n$. For any~$p\in [0,1]$,~$\{\DD_k\}$ is a sequence of~$p$-probabilistic ascent sets if and only if the sequence~$\{Y_k\}$ defined by~\eqref{eq:define_Y_k} satisfies
  \begin{equation}\label{eq:Y_k-condition}
    \Pr\left(Y_k = 0 \mid \FF_{k-1}\right) \;\ge\; p \quad \text{for each } k\ge 0.
  \end{equation}
\end{lemma}

Condition~\eqref{eq:Y_k-condition} is foundational to our analysis in
Subsections~\ref{ssec:markov_ineq} and~\ref{ssec:chernoff_bound}.
The properties of the sequence~$\{Y_k\}$ needed in our analysis
follow from~\eqref{eq:Y_k-condition} without relying on the specifics of
Algorithm~\ref{alg:direct_search_r}. Such properties include inequality~\eqref{eq:expc} and the
propositions in Subsection~\ref{ssec:lemmas}.

\subsection{Key ingredients and framework of our analysis}\label{ssec:key_ingredients}

Based on the sequence~$\{Y_k\}$ defined in~\eqref{eq:define_Y_k}, we now introduce three additional
sequences that will play major roles in our non-convergence study. We then present a lemma
and build the framework of our analysis around it.

Let us define
\begin{equation}
    \label{eq:YUE}
     \overline{Y}_k  \;=\; \frac{1}{k}\sum_{\ell=0}^{k-1} Y_\ell,
     \quad
     U_k  \;=\; \prod_{\ell=0}^{k-1}\gamma^{Y_\ell} \theta^{1 - Y_\ell},
     \quad
     E_k  \;=\; \bigcap_{\ell=0}^{k -1}\{Y_\ell = 0\},
\end{equation}
with the convention that
\begin{equation}
    \label{eq:UYE0}
  \overline{Y}_0 \;=\;  0, \quad
  U_0 \;=\;  1, \quad
  E_0 \;=\;  \Omega.
\end{equation}
Note that~$\{E_k\}$ is a nonincreasing sequence of events. In addition, since~$0 < \theta < \gamma$, we have
\begin{equation}
    \label{eq:EU}
    E_k \;=\;  \bigcap_{\ell=0}^{k -1} \{U_\ell = \theta^\ell\}.
\end{equation}
We can check that~$Y_k$ is~$\FF_k$-measurable,
and consequently,~$\overline{Y}_k$,~$U_k$, and~$E_k$ are~$\FF_{k-1}$-measurable.

Assuming that~$f$ is differentiable and convex,\footnote{\,Note that a real-valued differentiable convex function
is continuously differentiable~\cite[Theorem~25.5]{Rockafellar_1970}.}
Lemma~\ref{lem:X_k-to-U_k} links the iterates~$\{X_k\}$ with the sequences~$\{Y_k\}$ and~$\{U_k\}$.
As will be detailed in the proof, the convexity of~$f$ provides a useful connection between~$Y_k$
and iteration~$k$ of Algorithm~\ref{alg:direct_search_r}: if~$Y_k = 0$, then the descent
condition~\eqref{eq:sufficient_decrease} cannot be satisfied, leading to~$X_{k+1} = X_k$ and~$A_{k+1}
= \theta A_k$, which is essentially why the lemma holds.

\begin{lemma}\label{lem:X_k-to-U_k}
  Consider Algorithm~\ref{alg:direct_search_r} with~$f$ being differentiable and convex on~$\RR^n$. Then
  \begin{equation}\label{eq:X_k-to-U_k}
    \sup_{k\ge 0} \|X_k - x_0\|  \;\le\; \alpha_0 \sum_{k=0}^\infty Y_k U_k \;\le\; \alpha_0 \sum_{k=0}^\infty U_k.
  \end{equation}
\end{lemma}
\begin{proof}
  For each~$k\ge 0$, we note that
  \begin{equation}\label{eq:change_X_k}
  \|X_{k+1} - X_k\| \;\le\; Y_k A_k.
  \end{equation}
  Indeed, if~$Y_k = 0$, then~$\DD_k$ contains no descent direction, so that the descent condition~\eqref{eq:sufficient_decrease} can never be satisfied due to the convexity of~$f$, leading to~$X_{k+1} = X_k$ and thus~\eqref{eq:change_X_k}; when~$Y_k = 1$, inequality~\eqref{eq:change_X_k} holds because of our blanket assumption that~$\DD_k$ contains only unit vectors.
  Following a similar logic, we have
  \begin{equation}\label{eq:alpha}
    A_{k+1} \;\le\; \gamma^{Y_{k}}\theta^{1 - Y_{k}} A_k,
  \end{equation}
  which is because~$A_{k+1} = \theta A_k$ if~$Y_k = 0$ and~$A_{k+1} \le \gamma A_k$ otherwise.
  Recalling~$A_0 = \alpha_0$ and the definition of~$U_k$ in~\eqref{eq:YUE}, we use~\eqref{eq:alpha}
  recursively and obtain
  \begin{equation}\label{eq:upperbound}
    A_k \;\le\; \alpha_0 \prod_{\ell=0}^{k-1}\gamma^{Y_\ell} \theta^{1 - Y_\ell} \;=\; \alpha_0 U_k.
  \end{equation}
  Since~$X_0 = x_0$, by~\eqref{eq:change_X_k} and~\eqref{eq:upperbound}, we have
  \begin{equation}\label{eq:bound_X_k}
    \|X_k - x_0\| \;\le\; \sum_{\ell=0}^{k-1} \|X_{\ell+1} - X_\ell\| \;\le\; \sum_{\ell=0}^{k-1} Y_{\ell} A_\ell \;\le\; \alpha_0 \sum_{\ell=0}^{k-1} Y_{\ell} U_\ell \;\le\; \alpha_0 \sum_{\ell=0}^{k-1} U_\ell,
  \end{equation}
  where the last inequality is because~$Y_\ell \le 1$. Finally, we get~\eqref{eq:X_k-to-U_k} by taking the supremum over~$k\ge 0$ in~\eqref{eq:bound_X_k}.
\end{proof}

\begin{remark}
    \label{rem:simple_decrease}
    Lemma~\ref{lem:X_k-to-U_k} remains valid if Algorithm~\ref{alg:direct_search_r} adopts the simple
    decrease condition~\eqref{eq:simple_decrease} in place of the sufficient decrease
    condition~\eqref{eq:sufficient_decrease}, because the former still cannot be fulfilled
    if~$f$ is convex and~$Y_k = 0$.
    Hence, our non-convergence theory based on Lemma~\ref{lem:X_k-to-U_k} will also apply to
    the variant of Algorithm~\ref{alg:direct_search_r} with~\eqref{eq:simple_decrease}
    replacing~\eqref{eq:sufficient_decrease}.
    This includes Theorems~\mbox{\ref{thm:loose}--\ref{thm:quantitative}},~\ref{thm:tight},~\ref{thm:nonsmooth},
    and their corollaries.
\end{remark}

\textbf{Framework of our analysis.}~Roughly, the major step of our analysis is to show that
\begin{equation}\label{eq:non-conv-goal}
\Pr\left(\sup_{k\ge 0} \|X_k - x_0\| < \zeta\right) \;>\; 0
\end{equation}
for some~$\zeta > 0$; once this is established, we will have~$\{X_k\}$
bounded away from stationarity with positive probability if~$x_0$ is
``sufficiently non-stationary'',~\eg,~if~$\gap(x_0, \sol(f)) \ge \zeta$.
According to Lemma~\ref{lem:X_k-to-U_k}, we can establish~\eqref{eq:non-conv-goal} by
proving
\begin{equation}
  \label{eq:psum}
  \Pr\left(\sum_{k=0}^\infty Y_kU_k < \frac{\zeta}{\alpha_0}\right) \;>\; 0,
\end{equation}
or more strongly, $\Pr\big(\sum_{k=0}^\infty U_k < \zeta /\alpha_0\big) > 0$.
Our main results \mbox{Theorems~\ref{thm:loose}--\ref{thm:quantitative}} and~\ref{thm:tight} all
follow this framework, directly or indirectly.

\begin{remark}
    \label{rem:non_conv}
    According to~\eqref{eq:bound_X_k}, inequality~\eqref{eq:psum} indeed
    ensures~$\Pr\big(\!\sum_{k=0}^\infty\|X_{k+1}-X_{k}\|<\zeta\big) \!>\!0$, implying
    that
    \[
        \Pr\big(\{X_k\}~\text{converges to a point in}~\BB(x_0,\zeta)\big) \;>\; 0.
    \]
    Hence, the analysis sketched above will actually guarantee that~$\{X_k\}$ converges to
    a non-stationary point \textnormal{(}rather than diverges\textnormal{)} with positive probability.%
\end{remark}

\begin{remark}
    \label{rem:anyset}
    For any arbitrarily given set~$\TT$,
    inequality~\eqref{eq:psum} can ensure that~$\{X_k\}$ remains bounded away from~$\TT$ with positive
    probability as long as~$x_0$ is sufficiently distant from~$\TT$, although
    the primary interest of our analysis is the special case with~$\TT=\sol(f)$.
\end{remark}

\subsection{Non-convergence analysis via Markov's inequality}\label{ssec:markov_ineq}

We now conduct a non-convergence analysis of Algorithm~\ref{alg:direct_search_r} using Markov's
inequality. It serves as a quick illustration of the framework presented at the end of
Subsection~\ref{ssec:key_ingredients}. Its conclusion will be tightened by a refined
analysis in Subsection~\ref{ssec:chernoff_bound} via a conditional Chernoff bound.%

\begin{theorem}\label{thm:loose}
    Consider Algorithm~\ref{alg:direct_search_r} with~$f$ being differentiable and convex on~$\RR^n$. If~$\{\DD_k\}$~is a sequence of~$p$-probabilistic ascent sets with~$p > (\gamma - 1)/(\gamma-\theta)$, then we have
    \begin{equation}
        \label{eq:markov}
        \Pr\left(\gap\big(\{X_k\},\sol(f)\big)>0\right) \;>\; 0,
    \end{equation}
    provided that~$\gap(x_0, \sol(f)) > \alpha_0 / [1-\gamma(1-p) - \theta p]$.
\end{theorem}

\begin{proof}
    Note that
    \[
    \left\{\gap\big(\{X_k\},\sol(f)\big) > 0\right\}
    \,\supseteq\,
     \left\{\sup_{k\ge 0} \|X_k - x_0\| < \gap\big(x_0,\sol(f)\big)\right\}
     \,\supseteq\,
     \left\{\sum_{k=0}^\infty U_k < \frac{\gap\big(x_0,\sol(f)\big)}{\alpha_0}\right\},
    \]
    where the last inclusion is due to Lemma~\ref{lem:X_k-to-U_k}. Therefore, it suffices to show that
    \begin{equation}
        \label{eq:goal_markov}
      \Pr\left( \sum_{k=0}^\infty U_k \ge \frac{\gap\big(x_0,\sol(f)\big)}{\alpha_0} \right) \;<\; 1.
    \end{equation}
    Define~$\beta = 1/ [1-\gamma(1-p)-\theta p]$. Recalling the assumption
    that~$\gap(x_0, \sol(f)) > \alpha_0 \beta$ and Markov's inequality, we only need to prove that
    \begin{equation}\label{eq:to_prove_expectation}
      \expc\left( \sum_{k=0}^\infty U_k \right) \;\le\; \beta.
    \end{equation}
    Our assumption on~$p$ ensures~$0< \gamma(1-p) + \theta p < 1$, rendering~$\beta = \sum_{k=0}^\infty [\gamma(1-p) + \theta p]^k$.
    Meanwhile, Tonelli's theorem~\cite[page 420]{Royden_Fitzpatrick_2010} (also~\cite[Theorem~1.7.2]{Durrett_2019})
    yields~$\expc(\sum_{k=0}^\infty U_k) \!=\! \sum_{k=0}^\infty \expc(U_k)$. Thus, the proof of~\eqref{eq:to_prove_expectation} can be reduced to establishing
    \begin{equation}\label{eq:U_k-ineq_to_prove}
      \expc(U_k) \;\le\; [\gamma (1-p) + \theta p]^k \quad \text{for each~$k\ge 0$}.
    \end{equation}

    The proof of~\eqref{eq:U_k-ineq_to_prove} is standard. For each~$k\ge 0$,
    using the tower property of conditional expectation and the definition of~$\{U_k\}$ in~\eqref{eq:YUE}, we have
  \begin{equation*}
      \expc(U_{k+1})
      \;=\; \expc\left(\expc\left(\gamma^{Y_{k}}\theta^{1 - Y_{k}}U_k \mid \FF_{k-1} \right) \right)
      \;=\; \expc\left( \expc\left(\gamma^{Y_{k}}\theta^{1 - Y_{k}} \mid \FF_{k-1} \right) U_k \right),
  \end{equation*}
  where the last equality is because~$U_k$ is~$\FF_{k-1}$-measurable. By Lemma~\ref{lem:submartingale-Y_k},
  \begin{equation}
      \label{eq:expc}
  \expc\left(\gamma^{Y_{k}}\theta^{1 - Y_{k}} \mid \FF_{k-1} \right) = \gamma \Pr(Y_k = 1 \mid \FF_{k-1}) + \theta \Pr(Y_k = 0 \mid \FF_{k-1}) \le \gamma (1-p) + \theta p.
  \end{equation}
  Hence, we have~$\expc(U_{k+1}) \le [\gamma (1-p) + \theta p] \expc(U_k)$,
  which implies~\eqref{eq:U_k-ineq_to_prove} and concludes our proof.
\end{proof}

\begin{remark}
    Theorem~\ref{thm:loose} and its proof hold trivially if~$\sol(f)=\emptyset$, as we
    have~$\gap(\,\cdot\,,\sol(f)) = \infty$.%
\end{remark}

\begin{remark}
    \label{rem:quantitative_markov}
    Estimating the probability in~\eqref{eq:goal_markov} by Markov's inequality can
    strengthen~\eqref{eq:markov}~to
    \[
        \Pr\left(\gap\big(\{X_k\},\sol(f)\big)>0\right) \;\ge\;
        1 - \frac{\alpha_0}{[1-\gamma(1-p) - \theta p] \gap(x_0, \sol(f))}.
    \]
\end{remark}

\subsection{Non-convergence analysis via a conditional Chernoff bound}\label{ssec:chernoff_bound}

The analysis in the preceding subsection is based on the study of~$\sum_{k=0}^\infty U_k$ by
Markov's inequality, which is rather loose.
This subsection conducts a refined analysis via a conditional Chernoff bound
for~$\{Y_k\}$.
We establish the non-convergence of Algorithm~\ref{alg:direct_search_r}
when~$\{\DD_k\}$ is a sequence of~$p$-probabilistic ascent sets with~$p > p_*$
and~$x_0$ is any non-stationary point,
where~$p_*$ is defined by~\eqref{eq:define-p_*}.
Furthermore, we provide a lower bound for the probability that~$\{X_k\}$ remains
bounded away from stationarity and verify it numerically.
This analysis enables us to demonstrate that the probabilistic descent assumption
in Theorem~\ref{thm:conv_pds} is necessary for the global convergence of
Algorithm~\ref{alg:direct_search_r} with the typical implementation except for a boundary case.%

\subsubsection{Lemmas and observations}
\label{ssec:lemmas}
We now present several propositions about the $0$-$1$ process~$\{Y_k\}$ defined by~\eqref{eq:define_Y_k}
together with the associated sequences~$\{\overline{Y}_k\}$,~$\{U_k\}$, and~$\{E_k\}$ given by~\eqref{eq:YUE}.
We emphasize that the propositions are purely consequences of condition~\eqref{eq:Y_k-condition} and
independent of the algorithmic details of~PDS,
making them applicable in any context where~\eqref{eq:Y_k-condition} holds.
Indeed, they require only a weaker version of~\eqref{eq:Y_k-condition} that replaces~$\FF_k$
with~$\FF_k^{Y} =\sigma(Y_0, \dots, Y_k)$, which is a smaller~(coarser)~$\sigma$-algebra.%

Because of the clear separation between the algorithmic details and the propositions in this subsection,
which is done on purpose,
readers who are primarily interested in the algorithmic aspects of our analysis can proceed directly
to Subsection~\ref{sssec:qualitative_quantitative} and refer back to this subsection easily when needed.

Lemma~\ref{lem:strong_prob_exp_des} establishes a conditional Chernoff bound for~$\{Y_k\}$,
which is essentially a generalization of~\cite[Lemma~4.5]{Gratton_Royer_Vicente_Zhang_2015}.
Lemma~\ref{lem:shift} shows that condition~\eqref{eq:Y_k-condition} is preserved under conditioning
on~$E_{k_0}$ with any given integer~$k_0\ge 0$, as long as we shift the indices of~$\{Y_k\}$ and~$\{\FF_k\}$ by~$k_0$.
Both lemmas are proved in Appendix~\ref{app:proofs} since the arguments are straightforward.

\begin{lemma}\label{lem:strong_prob_exp_des}
  If~$0 < q < p \le 1$, then condition~\eqref{eq:Y_k-condition} implies that
  \begin{equation}\label{eq:chernoff_bound}
    \Pr \left( \overline{Y}_k \ge 1-q \mid E_{k_0} \right) \;\le\; \exp \left[-\frac{(p-q)^2}{2p}(k+k_0)\right] \quad \text{for all~$k\ge 0$~and~$k_0\ge 0$}.
  \end{equation}
\end{lemma}

\begin{remark}\label{rem:Pr_E_k}
    Noting the definition of~$E_k$ in~\eqref{eq:YUE},
  we can derive from condition~\eqref{eq:Y_k-condition} and the tower property of conditional expectations that
  \begin{equation*}
      \Pr(E_k ) \;\ge\; p^{k} \quad \text{for all~$k\ge 0$.}
  \end{equation*}
  Therefore, the conditional probability in Lemma~\ref{lem:strong_prob_exp_des} is well defined for
  any~$p > 0$.
\end{remark}

\begin{lemma}
    \label{lem:shift}
    Suppose that~$p>0$. Given an integer~$k_0\ge 0$,
    define~$\tilde{Y}_k \;=\; Y_{k_0+k}$ and~$\tilde{\FF}_{k} = \FF_{k_0+k}$ for each~$k$,
    and denote the probability measure~$\Pr(\,\cdot \mid E_{k_0})$ by~$\tilde{\Pr}(\cdot)$.
    Then condition~\eqref{eq:Y_k-condition} ensures~that%
    \begin{equation}
        \label{eq:Y_k_condition_tilde}
        \tilde{\Pr}(\tilde{Y}_k = 0 \mid \tilde{\FF}_{k-1}) \;\ge\; p \quad \text{for each}~k\ge 0.
    \end{equation}
\end{lemma}

Proposition~\ref{prop:cdf} is a key observation on the series~$\sum_{k=k_0}^\infty U_k$,
where~$k_0\ge 0$ is an integer. It shows that condition~\eqref{eq:Y_k-condition} with~$p > p_*$
renders a lower bound for the cumulative distribution function of~$\sum_{k=k_0}^\infty U_k$
conditioned on~$E_{k_0}$.
More importantly, this lower bound is a positive-valued function independent of~$k_0$
after a suitable scaling.

\begin{proposition}
    \label{prop:cdf}
    If~$p > p_*$, then condition~\eqref{eq:Y_k-condition} implies that there exists a
    function~$\Upsilon$ satisfying
    \begin{equation}
        \label{eq:cdf}
        \Pr\left( \sum_{k=k_0}^\infty U_k < \frac{\theta^{k_0}\zeta}{1-\theta}  \MID E_{k_0} \right)
        \;\ge\; \Upsilon(\zeta) \;>\; 0
    \end{equation}
    for all~$\zeta > 1$ and~$k_0\ge 0$.
    Here, the function~$\Upsilon$ is determined by~$p$, $\theta$, and~$\gamma$.
\end{proposition}
\begin{proof}
    Our proof has two steps. First, identify a function~$\Upsilon$ fulfilling~\eqref{eq:cdf}
    for~$\zeta > 1$ and~$k_0=0$;
    second, prove that~$\Upsilon$ still works when we relax~$k_0$ to all nonnegative integers.

    \textbf{Step~1.}~Since~$E_{0} = \Omega$ as mentioned in~\eqref{eq:UYE0},
    this step is to determine a positive number~$\Upsilon (\zeta)$ for an arbitrarily given~$\zeta > 1$
    so that
    \begin{equation}
        \label{eq:P0F}
        \Pr(F) \;\ge\; \Upsilon(\zeta)
        \qquad \text{with}\qquad
        F \;=\; \left\{ \sum_{k=0}^\infty U_k < \frac{\zeta}{1-\theta} \right\}.
    \end{equation}
    To this end, we consider the event~$E_l$ defined in~\eqref{eq:YUE} and note that
    \begin{equation}
        \label{eq:P0}
        \Pr (F) \;\ge\; \Pr(F \cap E_l) \;=\;
        \Pr(F \mid E_l)\,
        \Pr(E_l)
    \end{equation}
    for each~$l\ge 0$.
    In the sequel, we will bound~$\Pr(F\mid E_l)$ and~$\Pr(E_l)$ from below,
    and select an~$l$ in such a way that~\eqref{eq:P0} yields a desired lower bound for~$\Pr(F)$.

    Due to the definition of~$F$ in~\eqref{eq:P0F} and
    the fact that~$E_l = \bigcap_{k=0}^{l-1}\{U_k=\theta^k\}$ mentioned in~\eqref{eq:EU}, it
    holds that
    \begin{equation}
        \label{eq:Pr_F_given_El}
        \Pr(F \mid E_l) \;=\; \Pr\left( \sum_{k=0}^{l-1} \theta^k + \sum_{k=l}^{\infty} U_k
        < \frac{\zeta}{1-\theta} \MID E_l \right) \;\ge\; \Pr\left( \sum_{k=l}^{\infty} U_k
        < \frac{\zeta-1}{1-\theta}\MID E_l \right),
    \end{equation}
    motivating us to bound~$\sum_{k=l}^\infty U_k$ from above.
    To do this, we define~$q = (p+p_*)/2$ and note that
  \begin{equation}\label{eq:U_k_sum_bound}
  \left\{\sum_{k=l}^{\infty} U_k < \sum_{k=l}^{\infty} \left( \gamma^{1-q}\theta^{q} \right)^k \right\}
  \;\supseteq\; \bigcap_{k=l}^{\infty}\left\{ U_k^{1/k} < \gamma^{1-q}\theta^{q} \right\}
  \;=\; \bigcap_{k=l}^\infty\left\{\overline{Y}_k < 1-q\right\},
  \end{equation}
  where the last step is because~$U_k^{1/k} = \gamma^{\overline{Y}_k}\theta^{1-\overline{Y}_k}$
  by the definitions of~$\overline{Y}_k$ and~$U_k$ in~\eqref{eq:YUE}.
  Thus,
  \begin{equation}\label{eq:lead_to_Zk}
    \begin{aligned}
      \Pr\left( \sum_{k=l}^{\infty} U_k < \sum_{k=l}^{\infty} \left( \gamma^{1-q}\theta^{q} \right)^k \MID E_l \right)
      & \;\ge\; 1- \Pr\left(\bigcup_{k=l}^\infty \left\{\overline{Y}_k\ge 1-q\right\} \MID E_l \right)\\
      & \;\ge\; 1- \sum_{k=l}^\infty \exp\left[- \frac{(p-q)^2}{2p}(k+l)\right],
    \end{aligned}
  \end{equation}
  which invokes Lemma~\ref{lem:strong_prob_exp_des} in the last step. Let~$l$ be the smallest
  nonnegative integer satisfying
    \begin{equation}
        \label{eq:l_condition}
            \sum_{k=l}^{\infty} \left( \gamma^{1-q}\theta^{q} \right)^k \;\le\; \frac{\zeta-1}{1-\theta}
            \qquad\text{and}\qquad
            \sum_{k=l}^\infty \exp\left[- \frac{(p-q)^2}{2p}(k+l)\right] \;\le\; \frac{1}{2}.
    \end{equation}
    Such an~$l$ exists because~$\gamma^{1-q}\theta^{q} < 1$ and~$(p-q)^2/(2p) > 0$ (observe
    that~$\gamma^{1-p_*}\theta^{p_*}=1$ and recall that~$0\le p_* < q < p$).~The first inequality
    in~\eqref{eq:l_condition} ensures that the right-hand side
    of~\eqref{eq:Pr_F_given_El} is no less than the left-hand side
    of~\eqref{eq:lead_to_Zk}, and the second inequality in~\eqref{eq:l_condition} guarantees
    that the right-hand side of~\eqref{eq:lead_to_Zk} is at least~$1/2$. Therefore,
    we can join~\eqref{eq:Pr_F_given_El} and~\eqref{eq:lead_to_Zk} to obtain
    \[
        \Pr(F \mid E_l) \;\ge\; \frac{1}{2}.
    \]
    Meanwhile, Remark~\ref{rem:Pr_E_k} renders~$\Pr(E_l) \ge p^l$.
    Hence, inequality~\eqref{eq:P0} validates~\eqref{eq:P0F} if we set
    \begin{equation}
        \label{eq:Upsilon}
        \Upsilon(\zeta) \;=\; \frac{p^l}{2},
    \end{equation}
    which is legitimate because~$l$ is fully determined by~$\zeta$
    when~$p$,~$\theta$, and~$\gamma$ are given.
    Thus,~$\Upsilon$ is a function that completes the first step of the proof.

    \textbf{Step~2.}~Now, we prove that the function~$\Upsilon$ found in the first step
    satisfies~\eqref{eq:cdf} for all~$\zeta > 1$ and~$k_0 \ge 0$. Fix an arbitrary~$k_0 \ge 0$.
    Define~$\tilde{\Pr}$,~$\{\tilde{\FF}_k\}$, and~$\{\tilde{Y}_k\}$ as in Lemma~\ref{lem:shift}.
    According to this lemma, condition~\eqref{eq:Y_k-condition} implies condition~\eqref{eq:Y_k_condition_tilde},
    which has exactly the same form as~\eqref{eq:Y_k-condition}, with~$\tilde{\Pr}$,~$\{\tilde{Y}_k\}$,
    and~$\{\tilde{\FF}_k\}$ corresponding to~$\Pr$,~$\{Y_k\}$, and~$\{\FF_k\}$, respectively.
    Therefore, repeating the proof for~\eqref{eq:P0F}, we can verify that~$\Upsilon$ fulfills
    \begin{equation}
        \label{eq:P0F_tilde}
        \tilde{\Pr}(\tilde{F}) \;\ge\; \Upsilon(\zeta)
        \qquad \text{with}\qquad
        \tilde{F} \;=\; \left\{ \sum_{k=0}^\infty \tilde{U}_k < \frac{\zeta}{1-\theta} \right\}
    \end{equation}
for all~$\zeta > 1$,
where~$\tilde{U}_k = \prod_{\ell=0}^{k-1} \gamma^{\tilde{Y}_\ell}\theta^{1-\tilde{Y}_\ell}$ for each~$k\ge 0$.
We will show that~\eqref{eq:P0F_tilde} is indeed~\eqref{eq:cdf}.
The definitions of~$\{\tilde{Y}_k\}$,~$\{U_k\}$, and~$E_{k_0}$~(see
Lemma~\ref{lem:shift} and~\eqref{eq:YUE}) imply that%
\begin{equation}
    \label{eq:UU_k}
\tilde{U}_k \;=\;  \prod_{\ell=k_0}^{k_0+k-1} \gamma^{Y_\ell}\theta^{1-Y_\ell}
\;=\; U_{k_0}^{-1}U_{k_0+k}
\;=\; \theta^{-k_0} U_{k_0+k}
\end{equation}
when~$E_{k_0}$ occurs. Plugging~$\tilde{\Pr}(\cdot) = \Pr(\,\cdot\mid E_{k_0})$
and~\eqref{eq:UU_k} into~\eqref{eq:P0F_tilde}, we have
\[
    \Upsilon(\zeta) \;\le\;
    \Pr(\tilde{F} \mid E_{k_0})
    \;=\; \Pr\left( \sum_{k=0}^\infty \theta^{-k_0} U_{k_0+k} < \frac{\zeta}{1-\theta} \MID E_{k_0} \right)
    \;=\; \Pr\left( \sum_{k=k_0}^\infty U_k < \frac{\theta^{k_0}\zeta}{1-\theta} \MID E_{k_0} \right),
\]
which matches~\eqref{eq:cdf} as desired. This finishes our proof.
\end{proof}

\begin{remark}
    \label{rem:cdf}
    Given an integer~$k_0 \ge 0$,
    condition~\eqref{eq:Y_k-condition} with~$p > p_*$ indeed ensures the equivalence
    \begin{equation*}
        \Pr\left( \sum_{k=k_0}^\infty U_k < \frac{\theta^{k_0}\zeta}{1-\theta} \MID E_{k_0} \right)
        \,>\, 0
        \qquad \Longleftrightarrow \qquad \zeta \,>\, 1.
    \end{equation*}
    The implication from right to left is due to Proposition~\ref{prop:cdf}, while the reverse
    implication holds because~$\sum_{k=k_0}^\infty U_k \ge \sum_{k=k_0}^\infty\theta^k = \theta^{k_0} /(1-\theta)$.
\end{remark}

Proposition~\ref{prop:cdf} leads to Proposition~\ref{prop:rate_of_CDF_YU}, a crucial observation
on the cumulative distribution function of~$\sum_{k=0}^\infty Y_k U_k$.
When~$\{Y_k\}$ fulfills condition~\eqref{eq:Y_k-condition} with~$p> p_*$, this distribution function
turns out to be positive everywhere on~$(0,\infty)$, and its tail at~$0^+$ decays no faster than
a power function with exponent~$\log p / \log \theta$. This observation will help us establish the
non-convergence result in Theorem~\ref{thm:qualitative} and derive a lower bound
for the probability of non-convergence in Theorem~\ref{thm:quantitative}.

\begin{proposition}\label{prop:rate_of_CDF_YU}
  For~$\zeta > 0$, define
  \begin{equation}
    \label{eq:CDF_YU}
    \Phi(\zeta) \;=\; \Pr\left( \sum_{k=0}^{\infty} Y_k U_k < \zeta \right).
  \end{equation}
  If~$p > p_*$, then condition~\eqref{eq:Y_k-condition} implies
  that~$\Phi(\zeta) >0$ for all~$\zeta > 0$, and
  that there exists a constant~$C > 0$ such that
  \begin{equation}
    \label{eq:rate_of_CDF_YU}
    \Phi(\zeta) \;\ge\; C\,\zeta^{\frac{\log p}{\log \theta}} \quad\text{for}\;~\zeta \in (0,1).
  \end{equation}
\end{proposition}

\begin{proof}
    It suffices to prove~\eqref{eq:rate_of_CDF_YU}, which will ensure the positivity of~$\Phi(\zeta)$
    for all~$\zeta > 0$ because~$\Phi$ is nondecreasing.
    Given a~$\zeta \in (0,1)$, define
    \begin{equation}
        \label{eq:l}
        l \;=\; \left\lceil \frac{\log [\zeta (1-\theta)/2]}{\log \theta} \right\rceil.
    \end{equation}
    Then~$l\ge 0$.
    Recalling that~$E_l=\bigcap_{k=0}^{l-1}\{Y_k=0\}$ as defined in~\eqref{eq:YUE},
    we have
    \begin{equation}
        \label{eq:Y_kU_k}
        \left\{ \sum_{k=0}^\infty Y_kU_k < \zeta \right\}
        \;\supseteq\; \left\{ \sum_{k=l}^{\infty} Y_k U_k < \zeta\right\} \cap E_l
        \;\supseteq\; \left\{ \sum_{k=l}^{\infty} U_k < \frac{2\theta^l}{1-\theta} \right\} \cap E_l,
    \end{equation}
    where the last inclusion uses the inequality~$Y_k\le 1$ and the fact
    that~$2\theta^l/(1-\theta) \le \zeta$ by the definition~\eqref{eq:l} of~$l$.
    Combining~\eqref{eq:Y_kU_k} with the definition of~$\Phi$ in~\eqref{eq:CDF_YU}, we obtain
    \begin{equation}
        \nonumber
        \Phi(\zeta)
        \;\ge\; \Pr\left( \sum_{k=l}^{\infty} U_k < \frac{2\theta^l}{1-\theta} \MID E_l \right) \Pr(E_l)
        \;\ge\; \Upsilon(2) p^l,
    \end{equation}
    where~$\Upsilon(2)$ in the last step comes from Proposition~\ref{prop:cdf} and~$p^l$ comes from Remark~\ref{rem:Pr_E_k}.
    Therefore,
    \[
        \begin{split}
    \log\left[\Phi(\zeta)\zeta^{-\frac{\log p}{\log \theta}}\right]
    \,\ge \, \log[\Upsilon(2)] + l \log p - \left(\frac{\log p}{\log \theta}\right) \log \zeta
    \,=\, \log[\Upsilon(2)] +\left( l - \frac{\log \zeta}{\log \theta} \right)\log p.
        \end{split}
    \]
    Plugging the definition~\eqref{eq:l} of~$l$ into this inequality, we obtain by direct
    calculation that
    \begin{equation}
        \label{eq:log_Phi_zeta}
        \log\left[\Phi(\zeta)\zeta^{-\frac{\log p}{\log \theta}}\right]
        \;\ge\; \log[\Upsilon(2)] + \left(\frac{\log[(1-\theta)/2]} {\log \theta}-1\right)\log p,
    \end{equation}
    which implies~\eqref{eq:rate_of_CDF_YU}, with~$C$ being the
    exponential of the right-hand side in~\eqref{eq:log_Phi_zeta}.
\end{proof}

Note that the exponent in the lower bound in~\eqref{eq:rate_of_CDF_YU} is
independent of~$\gamma$. However, the lower bound itself does depend on~$\gamma$ through the
constant~$C$, which is increasing in~$\Upsilon(2)$ and thus decreasing in~$\gamma$ by
the proof of Proposition~\ref{prop:cdf} \textnormal{(}see~\eqref{eq:l_condition}
and~\eqref{eq:Upsilon} in particular\textnormal{)}.

\subsubsection{Qualitative and quantitative non-convergence results}
\label{sssec:qualitative_quantitative}

This subsection presents our main results on the non-convergence of Algorithm~\ref{alg:direct_search_r}.
We characterize the non-convergence of the algorithm qualitatively in Theorem~\ref{thm:qualitative} and
quantitatively in Theorem~\ref{thm:quantitative}, the latter providing a lower bound for the
probability of non-convergence.

It is worth stressing that our results allow us to use a
lower semicontinuous function~$\mu$ to measure the distance of a given point to optimality,
examples of such an optimality measure include~$f(\cdot)-\inf f$,~$\gap(\,\cdot\,,\,\sol(f))$, and~$\|\nabla f(\cdot)\|$.

Theorem~\ref{thm:qualitative} is our qualitative non-convergence result, stating that
Algorithm~\ref{alg:direct_search_r} stays away from the optimal set with positive probability under a probabilistic ascent
assumption, provided that the algorithm is initialized at a non-stationary point.

\begin{theorem}\label{thm:qualitative}
  Consider Algorithm~\ref{alg:direct_search_r} with~$f$ being differentiable and convex on~$\RR^n$.
  Suppose that~$\{\DD_k\}$~is a sequence of~$p$-probabilistic ascent sets with~$p > p_*$. Then we
  have
  \begin{equation}\label{eq:non-convergence_optimality}
      \Pr\left(\inf_{k\ge 0}\, \mu(X_k) > 0\right) \;>\; 0
  \end{equation}
  for any function~$\mu\mathrel{:}\RR^n\to (-\infty,\infty]$ that is lower semicontinuous
  with~$\mu(x_0)>0$.
  In particular, the conclusion holds if~$\mu$ is~$f(\cdot)-\inf f$,~$\gap(\,\cdot\,,\,\sol(f))$,
  or~$\|\nabla f(\cdot)\|$ and~$x_0$ is not stationary.
\end{theorem}

\begin{proof}
  Take a positive constant~$\varepsilon < \mu(x_0)$.
  By the lower semicontinuity of~$\mu$, there exists a~$\delta> 0$
  such that~$\{x \mathrel{:} \mu(x) > \varepsilon\} \supseteq \BB(x_0,\delta)$. Hence,
  \begin{equation}\label{eq:mu_semicontinuity}
    \left\{ \inf_{k\ge 0} \,\mu(X_k) > 0 \right\}
    \;\supseteq\;
    \Big\{\{X_k\} \subseteq \{x \mathrel{:} \mu(x) > \varepsilon\} \Big\}
    \;\supseteq\;
    \Big\{\{X_k\} \subseteq \BB(x_0, \delta) \Big\}.
  \end{equation}
  Meanwhile, Lemma~\ref{lem:X_k-to-U_k} implies that
    \begin{equation}\label{eq:prob_non-convergence}
        \Big\{ \{X_k\} \subseteq \BB(x_0, \delta) \Big\}
        \;\supseteq\; \left\{ \sup_{k\ge 0} \|X_k - x_0\| < \delta \right\}
        \;\supseteq\; \left\{\sum_{k=0}^\infty Y_k U_k < \frac{\delta}{\alpha_0} \right\}.
    \end{equation}
    The last event in~\eqref{eq:prob_non-convergence} has a positive probability by
    Proposition~\ref{prop:rate_of_CDF_YU}, because~$\{Y_k\}$ satisfies condition~\eqref{eq:Y_k-condition}
    according to
    Lemma~\ref{lem:submartingale-Y_k}. Therefore,~\eqref{eq:mu_semicontinuity}
    and~\eqref{eq:prob_non-convergence} yield~\eqref{eq:non-convergence_optimality}.
\end{proof}

\begin{remark}
  \label{rem:anyfun}
  Theorem~\ref{thm:qualitative} holds even if the lower semicontinuous function~$\mu$ has no
  relation to stationarity, although we are primarily interested in the case where~$\mu(x) = 0$ if and
  only if~$x$ is a stationary point. The same remark applies to Theorem~\ref{thm:quantitative} below.
  See also Remark~\ref{rem:anyset}.
\end{remark}

  Theorem~\ref{thm:qualitative} is stronger than Theorem~\ref{thm:loose} in three aspects.
  First, Theorem~\ref{thm:qualitative} has a weaker requirement on~$p$ since~$p_* = (\log\gamma) / \log(\theta^{-1}\gamma) < (\gamma - 1)/(\gamma-\theta)$.
  Second, the optimality measure in Theorem~\ref{thm:qualitative} can be any lower semicontinuous
  function~$\mu$, whereas Theorem~\ref{thm:loose} applies only to~$\gap(\,\cdot\,,\,\sol(f))$.
  Third, even when $\mu(x)=\gap(x,\sol(f))$, the condition~$\mu(x_0)>0$ in Theorem~\ref{thm:qualitative}
  is weaker than~$\gap(x_0,\sol(f)) > \alpha_0/[1-\gamma(1-p) - \theta p]$ in Theorem~\ref{thm:loose}.

Theorem~\ref{thm:quantitative} is our quantitative non-convergence result, which estimates the
probability that the optimality measure in Theorem~\ref{thm:qualitative} remains close to its
initial value. This provides a lower bound for the non-convergence probability of
Algorithm~\ref{alg:direct_search_r} if its starting point is not stationary.

\begin{theorem}\label{thm:quantitative}
     Under the settings of Theorem~\ref{thm:qualitative},
     if we assume further that~$\mu$ is locally Lipschitz continuous at~$x_0$,
     then there exist constants~$C>0$ and~$\overline{\zeta}>0$ such that the function
    \begin{equation}\label{eq:Psi_def}
            \Psi(\zeta) \;=\; \Pr\Big( \inf_{k\ge 0} \mu(X_k) \,\ge\, (1-\zeta)\,\mu(x_0)\Big)
    \end{equation}
    satisfies
    \begin{equation}\label{eq:pds_quantitative_rate}
        \Psi(\zeta) \;\ge\; C\,\zeta^{\frac{\log p}{\log \theta}}
        \quad\text{for}\;~\zeta \in (0, \overline{\zeta}).
    \end{equation}
\end{theorem}

\begin{proof}
  By our assumption on~$\mu$, there exist constants~$L>0$ and~$\delta>0$ such that
  \begin{equation}\label{eq:mu_Lipschitz}
    |\mu(x) - \mu(x_0)| \;\le\; L\|x-x_0\| \quad\text{for all~$x\in \BB(x_0, \delta)$}.
  \end{equation}
  For all~$\zeta \in (0,\, L\delta /\mu(x_0))$, combining~\eqref{eq:mu_Lipschitz} with Lemma~\ref{lem:X_k-to-U_k} renders
  \begin{equation*}
      \left\{ \inf_{k\ge 0} \mu(X_k) \,\ge\, (1-\zeta) \, \mu(x_0) \right\}
    \;\supseteq\;
    \Bigg\{\sup_{k\ge 0}\|X_k - x_0\| < \frac{\zeta\mu(x_0)}{L}\Bigg\}
    \;\supseteq\;
    \Bigg\{\sum_{k=0}^\infty Y_k U_k < \frac{\zeta\mu(x_0)}{L \alpha_0}\Bigg\}.
  \end{equation*}
  Consequently, the definition of~$\Phi$ in~\eqref{eq:CDF_YU} and that of~$\Psi$
  in~\eqref{eq:Psi_def} yield
  \[
      \Psi(\zeta)\;\ge\;\Phi\!\left(\frac{\zeta\mu(x_0)}{L\alpha_{0}}\right).
  \]
  Hence, Proposition~\ref{prop:rate_of_CDF_YU} guarantees the existence of~$C$ and~$\overline{\zeta}$
  that validate~\eqref{eq:pds_quantitative_rate}.
\end{proof}

Recall Corollary~\ref{cor:conv_pds}, which states that
Algorithm~\ref{alg:direct_search_r} will converge with probability~1
if~$m > \log_2(1-{\log \theta}/{\log \gamma})$.
We can now show the non-convergence side.

\begin{corollary}\label{cor:nonconv_pds}
  Consider Algorithm~\ref{alg:direct_search_r} with~$f$ being differentiable and convex on~$\RR^n$.
  Let~$\{\DD_k\}$ be defined as in Corollary~\ref{cor:conv_pds}.
  If~$\gamma = 1$ or
  \begin{equation}
      \label{eq:mineq}
      m \;<\; \log_2\left(1-\frac{\log \theta}{\log \gamma}\right),
  \end{equation}
  then~\eqref{eq:non-convergence_optimality} holds
  for any function~$\mu\mathrel{:}\RR^n\to (-\infty,\infty]$ that is lower semicontinuous with~$\mu(x_0)>0$.
  If we further assume that~$\mu$ is locally Lipschitz continuous at~$x_0$, then there exist
  constants~$C>0$ and~$\overline{\zeta}>0$ such that~\eqref{eq:pds_quantitative_rate} holds
  with~$p = 2^{-m}$.
\end{corollary}

\begin{proof}
  Proposition~\ref{prop:prob-asc} ensures that~$\{\DD_k\}$~is a sequence of~$p$-probabilistic
  ascent sets with~$p = 2^{-m}$. According to Theorems~\ref{thm:qualitative} and~\ref{thm:quantitative},
  it suffices to show that~$2^{-m} > p_*$, which is guaranteed
  by the definition of~$p_*$ in~\eqref{eq:define-p_*} if~$\gamma =1$ or~$m$
  satisfies~\eqref{eq:mineq}.
\end{proof}

\begin{remark}
    \label{rem:complement}
  Comparing Corollaries~\ref{cor:conv_pds} and~\ref{cor:nonconv_pds}, we observe that their
  requirements on the algorithmic parameters~$\theta$,~$\gamma$, and~$m$~are almost the
  negations of each other. The only gap
  is the boundary situation with~$m = \log_2 (1 - \log \theta / \log \gamma)$,
  which is a concern only if~$\log_2 (1 - \log \theta / \log \gamma)$ happens to be an integer.
  Neither the convergence theory in~\cite{Gratton_Royer_Vicente_Zhang_2015} nor our non-convergence
  theory covers this situation, leaving an interesting topic for future research.
  The one-dimensional case is however a trivial exception,
  because~$\{\DD_k\}$ is a sequence of~$p_0$-probabilistic~$1$-descent
  sets if~$m = \log_2 (1 - \log \theta / \log \gamma)$,
  making Algorithm~\ref{alg:direct_search_r} converge
  according to Theorem~\ref{thm:conv_pds}.
\end{remark}

For the typical implementation of Algorithm~\ref{alg:direct_search_r} specified in
Corollary~\ref{cor:conv_pds}, we are now able to present the equivalence between the global
convergence of the algorithm and the probabilistic descent property of~$\{\DD_k\}$,
excluding the situation where~$\log_2 (1 - \log \theta / \log \gamma)$ is an integer.
\begin{theorem}
    \label{thm:equiv}
    Consider Algorithm~\ref{alg:direct_search_r} with~$f$ being differentiable, convex, and bounded
    below on~$\RR^n$, and with~$\nabla f$ being Lipschitz continuous on~$\RR^n$.
    Let~$\{\DD_k\}$ be defined as in Corollary~\ref{cor:conv_pds}.
    Suppose that~$x_0\notin \sol(f)$,~$\gamma>1$, and~$\log_2(1-\log\theta/\log\gamma)$ is not
    an integer.
    Then the following statements are equivalent.
    \begin{enumerate}
        \item\label{it:liminf} $\Pr(\liminf_k\|G_k\|=0) = 1$.
        \item\label{it:mbound} $m > \log_2(1- \log\theta /\log\gamma)$.
        \item\label{it:pd} $\{\DD_k\}$ is a sequence of~$p_0$-probabilistic~$\kappa$-descent sets for some~$\kappa > 0$.
    \end{enumerate}
\end{theorem}

\begin{proof}
    The implication~\ref{it:liminf}~$\Rightarrow$~\ref{it:mbound} is guaranteed by
    Corollary~\ref{cor:nonconv_pds} and the fact that~$m$ cannot
    equal~$\log_2(1- \log\theta /\log\gamma)$.
    The implication~\ref{it:mbound}~$\Rightarrow$~\ref{it:pd} is because
    of~\cite[Corollary~B.4]{Gratton_Royer_Vicente_Zhang_2015},
    and~\ref{it:pd}~$\Rightarrow$~\ref{it:liminf} is due to Theorem~\ref{thm:conv_pds}.
\end{proof}

For the sequence~$\{\DD_k\}$ in Corollary~\ref{cor:conv_pds}, if
the inequality~$\Pr(\cm(\DD_k, -G_k) \ge \kappa\mid \FF_{k-1}) \ge p_0$ in
Definition~\ref{def:p-k_descent} holds for one single~$k\ge 0$, then it holds for
all~$k$~(see~\cite[proof of Corollary~B.3]{Gratton_Royer_Vicente_Zhang_2015}).
This fact is essential for the implication~\ref{it:liminf}~$\Rightarrow$~\ref{it:pd} in
Theorem~\ref{thm:equiv}. A general~$\{\DD_k\}$, however, may ensure the convergence of
Algorithm~\ref{alg:direct_search_r} while satisfying the aforementioned inequality only for a set
of~$k$~(e.g., for~$k$ large enough), without being a sequence of~$p_0$-probabilistic~$\kappa$-descent
sets in the sense of Definition~\ref{def:p-k_descent}. The convergence analysis of
Algorithm~\ref{alg:direct_search_r} in such a scenario is beyond the scope of this paper, but
Subsection~\ref{ssec:weaker_ass} will discuss its non-convergence under a similar setting.

\subsubsection{Tightness of the ascent probability}
\label{sssec:tightness}

Theorem~\ref{thm:qualitative} requires~$\{\DD_k\}$ to be a sequence of~$p$-probabilistic
ascent sets with~$p>p_*$.
Such a requirement cannot be relaxed to~$p\ge p_*$,
which can be seen from the one-dimensional case mentioned in Remark~\ref{rem:complement}.
Indeed, if~$n = 1$ and~$m = \log_2 (1 - \log \theta / \log \gamma)$, then~$\{\DD_k\}$ is both
a sequence of~$p_*$-probabilistic ascent sets and a sequence of~$p_0$-probabilistic~$1$-descent
sets, the latter ensuring the convergence of Algorithm~\ref{alg:direct_search_r} by
Theorem~\ref{thm:conv_pds}.
Example~\ref{exp:converge_pds} extends this idea to general dimensions.
Note that the example defines~$\{\DD_k\}$ using gradient information, even though practical
implementations of Algorithm~\ref{alg:direct_search_r} are supposed to be derivative-free.

\begin{example}\label{exp:converge_pds}
  Consider Algorithm~\ref{alg:direct_search_r} with~$f$ being continuously differentiable and bounded below on~$\RR^n$, and~$\nabla f$ being Lipschitz continuous on~$\RR^n$.
  For each~$k\ge 0$, define
  \begin{equation*}
    \ddd_k \;=\;
    \begin{cases}
      G_k/\|G_k\|, & \text{if}~G_k\ne 0,\\
      d, & \text{otherwise},
    \end{cases}
  \end{equation*}
  where~$d$ is a fixed unit vector \textnormal{(}\eg, a coordinate vector\textnormal{)}.
  Then we set~$\DD_k = \{\xi_k \ddd_k\}$, where~$\xi_k$ is a random variable
  independent of~$\FF_{k-1}$, taking~$1$ and~$-1$ with probability~$p_*$~and~$p_0$,
  respectively~\textnormal{(}recall that~$p_*+p_0 = 1$\textnormal{)}.
  Note that
  \begin{equation*}
    \Pr\Big( \min_{\ddd \in \DD_k} \ddd^\trs G_k \ge 0 \Mid \FF_{k-1} \Big)
    \;=\; \Pr\big( \xi_k \ddd_k^\trs G_k \ge 0 \mid \FF_{k-1}\big)
    \;\ge\; \Pr(\xi_k = 1 \mid \FF_{k-1}) \;=\; p_*.
  \end{equation*}
  Hence, $\{\DD_k\}$ is a sequence of~$p_*$-probabilistic ascent sets according to Proposition~\ref{prop:prob-asc-def-equiv}.
  Similarly, one can check that~$\{\DD_k\}$ is a sequence of~$p_0$-probabilistic~$1$-descent sets
  according to Definition~\ref{def:p-k_descent}.
  Therefore, Theorem~\ref{thm:conv_pds} ensures that~$\Pr\left(\liminf_k \|G_k\|=0\right) = 1$.
\end{example}

Complementing Example~\ref{exp:converge_pds}, Proposition~\ref{prop:critical} clarifies why the
analysis framework described in Subsection~\ref{ssec:key_ingredients} is incapable of handling
the case with~$\{\DD_k\}$ being a sequence of~$p_*$-probabilistic ascent sets:~in this case,
it can happen that~$\sum_{k=0}^\infty Y_kU_k$ almost surely diverges to infinity, trivializing
the bounds in Lemma~\ref{lem:X_k-to-U_k}.
We will prove this proposition in Appendix~\ref{app:proofs}.
\begin{proposition}
    \label{prop:critical}
    If the sequence~$\{Y_k\}$ satisfies
    \begin{equation}
        \label{eq:critical_Y_k}
        \Pr(Y_k =0 \mid \FF_{k-1}) \;=\; p_* \quad \text{for each}~k\ge 0,
    \end{equation}
    then~$\Pr\left(\sum_{k=0}^\infty Y_k U_k = \infty\right) = 1$.
\end{proposition}

\subsection{The motivating examples revisited}
\label{ssec:revisit_examples}

\subsubsection{Numerical verification of the quantitative non-convergence result}
\label{sssec:empirical_rate}

In this subsection, we demonstrate the quantitative non-convergence result in
Theorem~\ref{thm:quantitative} numerically by a refined version of the experiment
in Subsection~\ref{sssec:numerical_illust}.

As an example, we will focus on the case
with~$\mu(x)=f(x)-\inf f$, which reduces the function~$\Psi$ defined in~\eqref{eq:Psi_def} to
\[
    \Psi(\zeta) \;=\; \Pr\Big(\inf_{k\ge 0} f(X_k) \,\ge\, (1-\zeta)\,f(x_0) + \zeta\, \inf \!f\Big).
\]
Theorem~\ref{thm:quantitative} shows that the tail of~$\Psi$ at~$0^+$ decays at a rate no faster
than~$\zeta^{{\log p}/{\log \theta}}$. Geometrically speaking, if we plot~$\Psi(\zeta)$
against~$\zeta$ on a log-log scale, the slope of the curve at~$0^+$ should be no more
than~${\log p}/{\log \theta}$,
which will be illustrated numerically by the following experiment.

We set up the objective function and the algorithm in the same way as in
Subsection~\ref{sssec:numerical_illust} except for the algorithmic parameters~$\theta$,~$\gamma$, and~$m$.
To ensure the representativeness of the results, instead of fixing~$(\theta, \gamma, m) = (1/4, 3/2, 2)$
as in Subsection~\ref{sssec:numerical_illust},
we randomly sample five values of the triple~$(\theta,\gamma,m)$ by the following scheme.
\begin{enumerate}
    \item Sample~$p_*$ and~$\theta$ uniformly from the intervals~$(0,9/20)$ and~$(1/4,3/4)$, respectively.
    \item Set~$\gamma = \theta^{p_* / (p_* - 1)}$ and~$m = \lfloor -\log_2 p_* - \mathrm{eps}\rfloor$,
        where~$\mathrm{eps}$ is the machine epsilon.
\end{enumerate}
This scheme ensures that inequality~\eqref{eq:mineq} holds. Hence, the function~$\Psi$ satisfies
inequality~\eqref{eq:pds_quantitative_rate} in Theorem~\ref{thm:quantitative}~(see
Corollary~\ref{cor:nonconv_pds}).

Given a sample of~$(\theta,\gamma,m)$, we perform~$N = 10^7$ independent runs of
Algorithm~\ref{alg:direct_search_r}.
The best (lowest) function value found in each run is denoted by~$f_{\text{best}}$. Then we define
\[
    \hat{\Psi}(\zeta) \;=\; \frac{1}{N} \cdot \big[\,\text{number of runs with}~f_{\text{best}}
    \,\ge\, (1-\zeta)\,f(x_0) + \zeta\,\inf\! f \,\big],
\]
which will be used as an empirical estimator for~$\Psi(\zeta)$. Note that~$\inf\!f = 0$ in our experiment.

Figure~\ref{fig:pds_rate} plots~$\log_{10}[\hat{\Psi}(\zeta)]/(\log p/\log\theta)$
against~$\log_{10} \zeta$,
with~$\zeta$ varying in~$[10^{-3}, 10^{-1}]$.
Each curve in the figure corresponds to a sample of~$(\theta,\gamma,m)$.
Since we are concerned with the slopes rather than the intercepts, the curves are vertically
shifted by small constants to separate them visually.
As a reference, the figure includes a black dashed line with slope~$1$.

\begin{figure}[htbp!]
  \centering
  \includegraphics[width=0.65\textwidth]{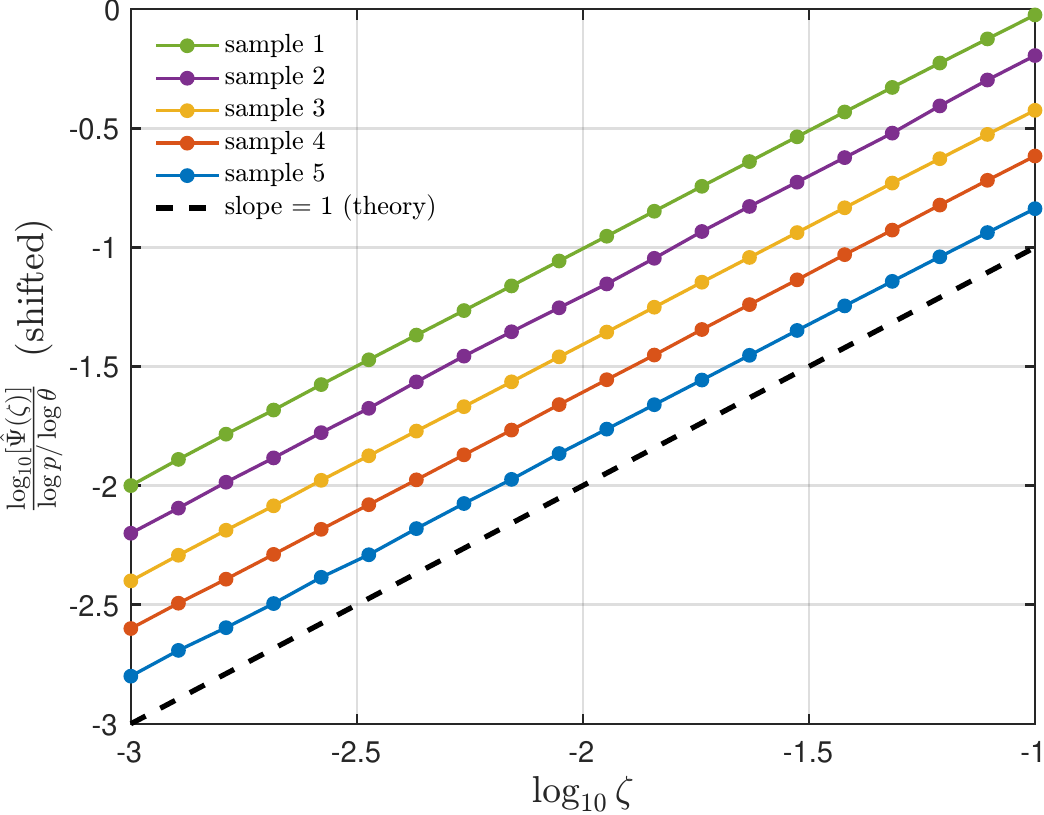}
  \caption{
      Curves of~$\log_{10}[\hat{\Psi}(\zeta)]/(\log p/\log\theta)$ versus~$\log_{10}\zeta$
      for five random samples of~$(\theta,\gamma,m)$.
      \mbox{The curves are vertically shifted for clarity. The dashed line is a reference
      line with slope~$1$.}
  }
  \label{fig:pds_rate}
\end{figure}

Admittedly, we have no guarantee about the quality of the
estimator~$\hat{\Psi}$ for~$\Psi$, and the interval~$[10^{-3}, 10^{-1}]$
is not necessarily a subset of the interval~$(0,\overline{\zeta})$ mentioned in
Theorem~\ref{thm:quantitative}.
Nevertheless, across all the samples, the curves in Figure~\ref{fig:pds_rate} are almost perfectly parallel to the reference line, which is
consistent with the rate in Theorem~\ref{thm:quantitative}. Indeed, the theorem only indicates that the
slopes of the curves are no more than that of the reference line.
The \emph{surprisingly perfect} parallelism strongly
motivates us to conjecture that the rate in the theorem is tight, which is an interesting topic
for future research.

\subsubsection{Explanations about Example~\ref{exp:simple_illustrate}}
\label{sssec:explain_example}

Now we explain Example~\ref{exp:simple_illustrate}.
Recall that~$n=1$ and
note that the~$\{\DD_k\}$ in this example is a particular case of the~$\{\DD_k\}$ specified
in Corollaries~\ref{cor:conv_pds} and~\ref{cor:nonconv_pds} with~$m = 1$.

When~$\theta\gamma \ge 1$, we have~$\gamma > 1$ and~$\log_2 (1 - \log \theta / \log \gamma) \le 1=m$.
According to Corollary~\ref{cor:conv_pds} and Remark~\ref{rem:complement}, it holds
that~$\Pr(\liminf_{k}|X_k|=0) = 1$.
Note that~$\liminf_k|X_k|=0$ is equivalent to~$X_k\to 0$ in this example due to the monotonicity of
Algorithm~\ref{alg:direct_search_r}.

When~$\theta\gamma < 1$, we have~$\gamma = 1$ or~$\log_2 (1 - \log \theta / \log \gamma) > 1 =m$.
Therefore, Corollary~\ref{cor:nonconv_pds} ensures~$\Pr(\inf_{k\ge 0} |X_k|> 0) > 0$
and yields the lower bound in~\eqref{eq:simple_illustrate}.%

\subsection{Non-convergence under a weaker condition}\label{ssec:weaker_ass}

Example~\ref{exp:converge_pds} shows that we cannot weaken our assumption in Theorem~\ref{thm:qualitative} by replacing~$p>p_*$ with~$p\ge p_*$.
However, it is indeed possible to relax the definition of
probabilistic ascent to obtain a weaker condition that renders a weaker non-convergence result
compared with Theorem~\ref{thm:qualitative}.

Consider Algorithm~\ref{alg:direct_search_r} with~$f$ being differentiable and convex on~$\RR^n$. In place of probabilistic ascent, this subsection assumes that~$\{\DD_k\}$ satisfies
\begin{equation}\label{eq:liminf-asm_D}
    \Pr\Big( \liminf_{k\to\infty} \left\{\Pr\left(\cm\left(\DD_k,-G_k\right) \le 0 \mid \FF_{k-1} \right) \ge p \ind(G_k \ne 0) \right\}\Big) \;>\; 0.
\end{equation}
According to Proposition~\ref{prop:prob-asc-def-equiv}, condition~\eqref{eq:liminf-asm_D} holds if
and only if the sequence~$\{Y_k\}$ defined in~\eqref{eq:define_Y_k} fulfills
\begin{equation}\label{eq:liminf-asm}
  \Pr\Big( \liminf_{k\to\infty} \left\{\Pr\left(Y_k = 0 \mid \FF_{k-1}\right) \ge p \right\} \Big) \;>\; 0.
\end{equation}

\begin{remark}\label{rem:weaker_assumption}
  Condition~\eqref{eq:liminf-asm} means that the event
  \begin{equation}\label{eq:liminf-asm_event}
    \left\{\Pr(Y_k = 0 \mid \FF_{k-1}) \ge p~\text{for all sufficiently large}~k\right\}
  \end{equation}
  occurs with positive probability. This is weaker than
  \begin{equation}\label{eq:relax-asm-1}
      \Pr\left( \bigcap_{k=0}^\infty \left\{\Pr\left(Y_k = 0 \mid \FF_{k-1}\right) \ge p \right\} \right) \;>\; 0,
    \end{equation}
   which means that the event~$\left\{\Pr(Y_k = 0 \mid \FF_{k-1}) \ge p ~\text{for each}~k\right\}$~happens with positive probability. Condition~\eqref{eq:liminf-asm} is also weaker than
    \begin{equation}\label{eq:relax-asm-2}
      \sum_{k=0}^\infty \Pr\left( \left\{\Pr\left(Y_k = 0 \mid \FF_{k-1}\right) < p \right\}\right) \;<\; \infty,
    \end{equation}
    which implies that the event~\eqref{eq:liminf-asm_event} occurs \as~by the Borel--Cantelli Lemma~\cite[Theorem~2.3.1]{Durrett_2019}.
\end{remark}

As stated in Lemma~\ref{lem:submartingale-Y_k}, $\{\DD_k\}$ is a sequence of~$p$-probabilistic ascent sets if and only if the sequence~$\{Y_k\}$ satisfies
$\Pr\left(Y_k = 0 \mid \FF_{k-1}\right) \ge p$ for each~$k\ge 0$, a condition stronger than~\eqref{eq:liminf-asm}.
Indeed, such a condition ensures both~\eqref{eq:relax-asm-1} and~\eqref{eq:relax-asm-2}, either of which in turn implies~\eqref{eq:liminf-asm} as discussed in Remark~\ref{rem:weaker_assumption}.
Therefore, condition~\eqref{eq:liminf-asm_D} can be regarded as a relaxation of~$p$-probabilistic ascent defined in Definition~\ref{def:prob-ascent}.

Before showing the non-convergence result under condition~\eqref{eq:liminf-asm_D}, we introduce Lemma~\ref{lem:series_con}, which will be proved based on Lemma~\ref{lem:non-iid-LLN}, a strong law of large numbers for martingales.

\begin{lemma}[\cite{Chow_1967}]\label{lem:non-iid-LLN}
  Let~$\{W_k\}$ be a martingale. If there exists an~$\alpha \ge 1$ such that
  \begin{equation*}
    \sum_{k=1}^\infty \expc\left( |W_k - W_{k-1}|^{2\alpha} \right) / k^{1+\alpha} < \infty,
  \end{equation*}
  then we have~$W_k / k \to 0$ \as
  In particular,~$W_k / k \to 0$ \as~if~$\{W_k\}$ has bounded increments.
\end{lemma}

\begin{lemma}\label{lem:series_con}
  If~$p > p_*$, then condition~\eqref{eq:liminf-asm} implies that
    \begin{equation}\label{eq:series_convergence_result}
      \Pr\left( \sum_{k=0}^{\infty} U_k < \infty \right) \;>\; 0.
    \end{equation}
\end{lemma}

\begin{proof}
    By the root test, the series~$\sum_{k=0}^\infty U_k$ converges if~$\limsup_{k} U_k^{1/k} < 1$.
    Recalling the definitions of~$U_k$ and~$\overline{Y}_k$~in~\eqref{eq:YUE}
    as well as the fact that~$p_* = (\log \gamma)/\log (\theta^{-1}\gamma)$, we have
    \begin{equation}\label{eq:log_U_k}
      \log\left(U_k^{1/k}\right) \;=\; \log\left(\gamma^{\overline{Y}_k}\theta^{1 - \overline{Y}_k} \right) \;=\; \log\theta + \overline{Y}_k \log(\theta^{-1}\gamma) \;=\; [(p_* - 1) + \overline{Y}_k] \log(\theta^{-1}\gamma).
    \end{equation}
    Hence, it holds that
    \begin{equation}
        \label{eq:Y_k_limsup}
      \left\{\sum_{k=0}^{\infty} U_k < \infty\right\}
      \;\supseteq\; \left\{\limsup_{k\to\infty} \log\left(U_k^{1/k}\right) < 0\right\}
      \;=\; \left\{\limsup_{k\to\infty} \overline{Y}_k < 1 - p_*\right\},
    \end{equation}
    where the last step uses equality~\eqref{eq:log_U_k} and~$\log(\theta^{-1}\gamma)>0$.
    Therefore, by our assumption that~$p > p_*$, inequality~\eqref{eq:series_convergence_result}
    can be established by proving
    \begin{equation}\label{eq:limsup-need-to-show}
      \Pr\left( \limsup_{k\to\infty} \overline{Y}_k \le 1 - p  \right) \;>\; 0.
    \end{equation}

    To prove~\eqref{eq:limsup-need-to-show}, let us define
    \begin{equation*}
      P_k \;=\; \Pr\left(Y_k = 0 \mid \FF_{k-1}\right) \quad \text{for each } k\ge 0.
    \end{equation*}
    Then~$\expc(Y_k + P_k - 1 \mid \FF_{k-1}) = 0$ for each~$k$,
    so~$\big\{\sum_{\ell=0}^{k-1} (Y_\ell + P_\ell - 1)\big\}$ is a martingale with respect
    to~$\{\FF_{k-1}\}$. In addition, this martingale has bounded increments. Thus, Lemma~\ref{lem:non-iid-LLN} leads to
    \begin{equation*}
      \lim_{k\to\infty} \left(\overline{Y}_k + \overline{P}_k - 1\right) = 0 \quad \text{\as},
    \end{equation*}
    where~$\overline{P}_k = k^{-1}\sum_{\ell=0}^{k-1} P_\ell$.
    Hence,\footnote{\,Recall
    that~$\limsup_k a_k + \liminf_k b_k = \lim_k(a_k +b_k)$ for bounded real
    sequences~$\{a_k\}$ and~$\{b_k\}$ when the limit on the right-hand side exists.}
    we have~$\limsup_{k} \overline{Y}_k + \liminf_{k} \overline{P}_k=1$~\as, implying
    \begin{equation}\label{eq:limsup-need-to-show-2}
      \begin{aligned}
          \Pr\left( \limsup_{k\to\infty} \overline{Y}_k \,\le\, 1 - p  \right)
          \;=\; \Pr\left( \liminf_{k\to\infty} \overline{P}_k \,\ge\, p \right)
          \;\ge\; \Pr\left(\liminf_{k\to\infty} \,\{P_k \ge p \}\right).
      \end{aligned}
    \end{equation}
    The right-hand side of~\eqref{eq:limsup-need-to-show-2} is precisely the probability
    in condition~\eqref{eq:liminf-asm} and hence is positive.
    Therefore, inequality~\eqref{eq:limsup-need-to-show} is justified and
    the proof is complete.
\end{proof}

\begin{remark}\label{rem:stronger_assumption}
  In comparison with~\eqref{eq:series_convergence_result}, condition~\eqref{eq:Y_k-condition} with~$p>p_*$ implies
  \begin{equation}\label{eq:series_convergence_result_stronger}
    \Pr\left( \sum_{k=0}^{\infty} U_k < \infty \right) \;=\; 1.
  \end{equation}
  The proof is similar to that of Lemma~\ref{lem:series_con}. The major difference is
  that~\eqref{eq:Y_k-condition} ensures that the right-hand side
  of~\eqref{eq:limsup-need-to-show-2} equals~$1$, and hence~\eqref{eq:series_convergence_result_stronger}
  holds according to~\eqref{eq:Y_k_limsup}.
  In addition, by Proposition~\ref{prop:cdf}, condition~\eqref{eq:Y_k-condition} with~$p>p_*$
  implies that~$\Pr(\sum_{k=0}^\infty U_k < s) > 0$ for all~$s > 1/(1-\theta)$,
  whereas~\eqref{eq:series_convergence_result} only ensures that~$\Pr(\sum_{k=0}^\infty U_k < s) > 0$
  for some sufficiently large~$s$.
\end{remark}

Now, we are ready to present the non-convergence result under the weaker
condition~\eqref{eq:liminf-asm_D}. Its proof is similar to that of Theorem~\ref{thm:qualitative} with the help of Lemma~\ref{lem:series_con}.

\begin{theorem}\label{thm:tight}
    Consider Algorithm~\ref{alg:direct_search_r} with~$f$ being differentiable and convex on~$\RR^n$. If~$\{\DD_k\}$~satisfies~\eqref{eq:liminf-asm_D} with~$p>p_*$, then there exists a positive constant~$\zeta$~such that
    \begin{equation}
        \label{eq:weak_nonc}
        \Pr\left(\gap(\{X_k\},\sol(f))>0\right) \;>\; 0,
    \end{equation}
    provided that~$\gap(x_0,\sol(f)) \ge \zeta$.
\end{theorem}

\begin{proof}
  By Lemma~\ref{lem:series_con}, there exists a positive constant~$\zeta$~such that
  \begin{equation}
      \label{eq:weak_sumu}
    \Pr\left(\sum_{k=0}^\infty U_k < \frac{\zeta}{\alpha_0}\right) \;>\; 0.
  \end{equation}
  Meanwhile, when~$\gap(x_0,\sol(f)) \ge \zeta$, we have
  \[
      \left\{ \gap(\{X_k\},\sol(f))> 0 \right\}
      \;\supseteq\;
      \left\{ \sup_{k\ge 0} \|X_k -x_0\| < \zeta \right\}
      \;\supseteq\;
      \left\{ \sum_{k=0}^\infty U_k < \frac{\zeta}{\alpha_0} \right\},
  \]
  where the last inclusion is due to Lemma~\ref{lem:X_k-to-U_k}. Therefore,~\eqref{eq:weak_nonc}
  holds according to~\eqref{eq:weak_sumu}.
\end{proof}

\section{Extension to the nonsmooth case}\label{sec:nonsmooth}

In this section, we extend our non-convergence results to the nonsmooth case, assuming
that~$f\mathrel{:} \RR^n\to\RR$ is only convex but not necessarily differentiable. We will show that
the non-convergence theorems in Section~\ref{sec:nonconvergence} still hold if we
generalize Definition~\ref{def:prob-ascent} of probabilistic ascent
to Definition~\ref{def:prob-ascent-nonsmooth} below.

We use~$f^\circ(\,\cdot\,;d)$~to denote the Clarke generalized directional derivative
of~$f$ in a direction~$d\in\RR^n$, and~$\clarke f$~to denote the Clarke subdifferential
of~$f$~(see~\cite[Definitions~1.1 and~1.3]{Clarke_1975}).
When~$f$ is convex, they reduce to the usual (one-sided) directional derivative and
the subdifferential, respectively~\cite[Proposition~2.2.7]{Clarke_1990}, but we do not need
convexity until Theorem~\ref{thm:nonsmooth}.

\begin{definition}\label{def:prob-ascent-nonsmooth}
    Let $p \in [0,1]$.
    Consider Algorithm~\ref{alg:direct_search_r} with~$f$ being locally Lipschitz continuous on~$\RR^n$.
    The sequence~$\{\DD_k\}$~is said to be a sequence of~$p$-probabilistic ascent sets if it satisfies
  \begin{equation}\label{eq:prob-asc-def-nonsmooth}
    \Pr\Big(\min_{\ddd\in\DD_k} f^\circ(X_k;\ddd) \ge 0 \Mid \FF_{k-1} \Big)
    \;\ge\; p \quad \text{for each~$k\ge 0$}.
  \end{equation}
\end{definition}

\begin{remark}
  When~$f$ is continuously differentiable on~$\RR^n$, Definition~\ref{def:prob-ascent-nonsmooth}
  is equivalent to Definition~\ref{def:prob-ascent}. See Proposition~\ref{prop:prob-asc-def-equiv}
  and Remark~\ref{rem:prob_asc_alt} for details.
\end{remark}

Proposition~\ref{prop:prob-asc_nonsmooth} extends Proposition~\ref{prop:prob-asc}
to the nonsmooth case.
See Appendix~\ref{app:proofs} for its proof.

\begin{proposition}\label{prop:prob-asc_nonsmooth}
    Consider Algorithm~\ref{alg:direct_search_r} with~$f$ being locally Lipschitz continuous on~$\RR^n$.
    The sequence~$\{\DD_k\}$ specified in Proposition~\ref{prop:prob-asc} is a sequence
    of~$p$-probabilistic ascent sets as defined in Definition~\ref{def:prob-ascent-nonsmooth}
    with~$p = 2^{-m}$.
\end{proposition}

The~$0$-$1$ process~$\{Y_k\}$ defined in~\eqref{eq:define_Y_k} is fundamental
to our non-convergence analysis in the smooth case.
We now extend the definition of~$\{Y_k\}$ from~\eqref{eq:define_Y_k} to
\begin{equation}\label{eq:define_Y_k_nonsmooth}
    Y_k \;=\; \ind\!\left(\min_{\ddd\in\DD_k} f^\circ(X_k;\ddd) < 0\right)
    \quad \text{for each } k\ge 0.
\end{equation}
As in the smooth case,~$Y_k$ is~$\FF_k$-measurable for each~$k\ge 0$, because
the mapping~$(x,d)\mapsto f^\circ(x;d)$ is upper
semicontinuous~\cite[Proposition~2.1.1~(b)]{Clarke_1990} and hence Borel measurable.
It is clear that~$\{\DD_k\}$ is a sequence of~$p$-probabilistic ascent sets if and only
if~$\Pr(Y_k = 0 \mid \FF_{k-1}) \ge p$ for each~$k\ge 0$.
In other words,
Lemma~\ref{lem:submartingale-Y_k} remains valid under the new definition of~$\{Y_k\}$.

Now, we extend our non-convergence results in
Theorems~\ref{thm:qualitative} and~\ref{thm:quantitative} to the nonsmooth case.

\begin{theorem}\label{thm:nonsmooth}
  With~probabilistic ascent defined in Definition~\ref{def:prob-ascent-nonsmooth},
  Theorems~\ref{thm:qualitative} and~\ref{thm:quantitative} still hold if we remove the
  differentiability assumption about~$f$ and replace~$\|\nabla f(\cdot)\|$
  with~$\gap(0,\clarke f(\cdot))$.
\end{theorem}
\begin{proof}
  We only need to verify that Lemmas~\mbox{\ref{lem:submartingale-Y_k}--\ref{lem:shift}} and
  Propositions~\mbox{\ref{prop:cdf}--\ref{prop:rate_of_CDF_YU}} remain valid
  in the new setting.
  Lemma~\ref{lem:submartingale-Y_k} holds as mentioned above.
  For Lemma~\ref{lem:X_k-to-U_k}, we note that the convexity of~$f$ makes it impossible
  to satisfy the descent condition in Algorithm~\ref{alg:direct_search_r} when~$Y_k=0$,
  leading to inequalities~\mbox{\eqref{eq:change_X_k}--\eqref{eq:alpha}}, which in turn validate
  Lemma~\ref{lem:X_k-to-U_k}.
  Lemmas~\mbox{\ref{lem:strong_prob_exp_des}--\ref{lem:shift}} and
  Propositions~\mbox{\ref{prop:cdf}--\ref{prop:rate_of_CDF_YU}} are still true because they
  only rely on Lemma~\ref{lem:submartingale-Y_k} and the~$\FF_{k}$-measurability of~$Y_k$.
  The proof is complete.
\end{proof}

\begin{remark}
    \label{rem:optimality_measure}
    Theorem~\ref{thm:nonsmooth} allows the lower semicontinuous
    function~$\mu$ in~\eqref{eq:non-convergence_optimality} to
    be~$f(\cdot) - \inf f$,
    $\gap(\,\cdot\,,\sol(f))$, and~$\gap(0,\clarke f(\cdot))$.
    The lower semicontinuity of~$f(\cdot) - \inf f$ and~$\gap(\,\cdot\,, \sol(f))$ are trivial.
    That of~$\gap(0,\clarke f(\cdot))$ is also basic, but we present it as
    Lemma~\ref{lem:optimality_measure} for completeness.
\end{remark}

Theorem~\ref{thm:loose} also holds in the nonsmooth case since it is weaker than
Theorem~\ref{thm:qualitative}.
As for Theorem~\ref{thm:tight}, we can extend it to the nonsmooth case
by replacing condition~\eqref{eq:liminf-asm_D} with
\begin{equation}
\label{eq:liminf_asm_D_nonsmooth}
\Pr\left(\liminf_{k\to\infty}
    \Big\{\Pr\Big(\min_{\ddd\in\DD_k} f^\circ(X_k;\ddd) \ge 0 \Mid \FF_{k-1} \Big) \ge p \Big\}
\right) \;>\;  0,
\end{equation}
which is equivalent to condition~\eqref{eq:liminf-asm} after we switch the definition
of~$\{Y_k\}$ to~\eqref{eq:define_Y_k_nonsmooth}.
Lemma~\ref{lem:series_con} still holds, leading to an analogue of
Theorem~\ref{thm:tight} with condition~\eqref{eq:liminf-asm_D} changed to~\eqref{eq:liminf_asm_D_nonsmooth}.

Hence, we have extended our non-convergence theory of PDS to the nonsmooth case.

\section{Conclusions and perspectives}\label{sec:conclusion}

We have established the non-convergence theory of PDS~(Algorithm~\ref{alg:direct_search_r}).
For convex objectives, if the polling direction sets satisfy the~$p$-probabilistic ascent
condition~(Definition~\ref{def:prob-ascent}) with
\[
   p \;>\; p_* \;=\;  \frac{\log \gamma}{\log(\theta^{-1}\gamma)},
\]
then the iterates of PDS remain bounded away from the solution set with positive
probability unless the starting point is a solution~(Theorem~\ref{thm:qualitative}).
More significantly, we provide a lower bound on this probability~(Theorem~\ref{thm:quantitative}),
and our numerical experiments suggest that the bound is sharp.

For the typical implementation of PDS, where each polling set consists of~$m$~\iid~random
directions uniformly distributed on the unit sphere
with no dependence on existing polling directions or iterates,
the above result implies that
the algorithm is not globally convergent for convex objectives if~$\gamma=1$
or~$m<\log_2\left(1-{\log\theta}/{\log\gamma}\right)$, which is the
opposite of the known convergence condition unless~$m = \log_2(1 - \log\theta/\log\gamma)$.
This confirms that the probabilistic descent condition in the existing convergence theory is
essential for the typical implementation~(Theorem~\ref{thm:equiv}).

It is also noteworthy that the condition $p>p_*$ for non-convergence cannot be relaxed
to~$p\ge p_*$~(Example~\ref{exp:converge_pds}), although a weaker non-convergence condition does
exist~(Theorem~\ref{thm:tight}). In addition,
our non-convergence theory covers the nonsmooth case by replacing the probabilistic ascent
condition with a natural generalization based on the Clarke subdifferential~(Section~\ref{sec:nonsmooth}).

Two technical problems remain open. The first one is whether the aforementioned typical implementation
of PDS converges in the borderline scenario with~$m = \log_2(1 - \log\theta/\log\gamma)$. The second is the
sharpness of the lower bound on the non-convergence probability in Theorem~\ref{thm:quantitative}.

The most intriguing direction for future research, however, extends far beyond PDS. As mentioned in
Section~\ref{sec:introduction}, we are interested in developing analogous non-convergence theories for other randomized methods, including
probabilistic trust region~\cite{Bandeira_Scheinberg_Vicente_2014,
Gratton_Royer_Vicente_Zhang_2015,Wang_Yuan_2022},
line search~\cite{Cartis_Scheinberg_2018,Berahas_Cao_Scheinberg_2021}, cubic
regularization~\cite{Cartis_Scheinberg_2018}, and subspace methods~\cite{Cartis_Roberts_2022,Roberts_Royer_2023}.
Similar to the probabilistic descent condition~\eqref{eq:p-k_descent-def},
the convergence analysis of these methods relies on submartingale-like assumptions in the form of
\[
    \Pr\Big(F_k \Mid \GG_{k-1}\Big) \;\ge\; P_0 \quad \text{for each~$k\ge 0$},
\]
where~$F_k$ is an event favorable to the success of the algorithm at iteration~$k$,~$\GG_{k-1}$ is
a~$\sigma$-algebra representing the history of the algorithm up to iteration~$k-1$, and~$P_0$
is a positive constant.
How essential are these assumptions for the convergence?
Answering this question would deepen our understanding of these methods and of submartingale-like
assumptions in general.
We are currently working on the corresponding non-convergence theories, with the aim of providing a
comprehensive answer to this question.
We hope that systematic non-convergence analysis of this kind, in contrast to the
construction of particular non-convergence examples, will receive more attention in the study of
optimization algorithms.

Defined according to the step size contracting-expanding scheme of PDS, the series
\[
    \sum_{k=0}^\infty \prod_{\ell=0}^{k-1} \gamma^{Y_\ell} \theta^{1-Y_\ell}
    \quad\text{and}\quad
    \sum_{k=0}^\infty Y_k\prod_{\ell=0}^{k-1} \gamma^{Y_\ell} \theta^{1-Y_\ell}
\]
have played a vital role in this paper.
They control the distance between the iterates and the starting point, and their boundedness is closely
related to the non-convergence of PDS.
The above-mentioned randomized methods all contain contracting-expanding schemes for certain
algorithmic parameters. Hence,
we expect that similar series will be instrumental in their non-convergence analysis,
where the tools and results developed here about these series will be a useful guide.
This is particularly true of the conditional Chernoff bound and its consequences in
Subsection~\ref{ssec:lemmas}, which are independent of the algorithmic details of PDS.
Moreover, these series can indeed provide a new way of establishing the global convergence of a class of
randomized methods including PDS, which is another topic we are currently studying but is beyond the
scope of this paper.
Finally, the stochastic processes defined by these series appear to have interesting probabilistic structures.
Further study of them may reveal additional connections between probability theory and numerical
optimization. This paper represents only a small step in this direction.

\small

\section*{Acknowledgments}
This paper is part of the PhD thesis of Cunxin Huang,
co-supervised by Zaikun Zhang and Professor Xiaojun Chen from The Hong
Kong Polytechnic University. Both authors are grateful to Professor Chen for her help
during Huang's PhD study.
\section*{Declarations}

\textbf{Funding.} This work was supported by the National Key R\&D Program of China (2023YFA1009300),
the General Program of the National Natural Science Foundation of China~(12571335), and the Hong
Kong PhD Fellowship Scheme~(PF21-56718).

\noindent\textbf{Conflict of interest.} The authors declare that they have no conflict of interest.

\ifnum\value{cite}>0
    \bibliography{\bibfile}
    \bibliographystyle{plain}
\fi

\begin{appendices}

\section{Basic lemmas}\label{app:basic_lemmas}

This section collects some basic lemmas needed in our analysis.
They are entirely independent of our algorithmic discussion.
We begin with Lemma~\ref{lem:pef}, a direct consequence of~\cite[Theorem~4.1.14]{Durrett_2019}.

\begin{lemma}
    \label{lem:pef}
    If~$E$ and~$F$ are events, and~$\GG$ is a $\sigma$-algebra with~$F\in\GG$,
    then~$\Pr(EF \mid \GG) = \Pr(E\mid \GG) \ind(F)$.
\end{lemma}

Lemma~\ref{lem:prob_sigma_to_event} presents a basic connection between the conditional probability with respect to
a~$\sigma$-algebra and that with respect to an event.
\begin{lemma}\label{lem:prob_sigma_to_event}
  Let~$E$ be an event and~$\GG$ be a~$\sigma$-algebra. Then~$\Pr\left( E \mid \GG \right) \;\ge\; p$ if and only if~$\Pr\left( E \mid F \right) \ge p$ for all~$F\in\GG$ with~$\Pr(F) > 0$.
\end{lemma}
\begin{proof}
For all~$F\in\GG$ with~$\Pr(F) > 0$,
the law of total probability and Lemma~\ref{lem:pef} yield
\begin{equation}
    \label{eq:cond-prob-sigma}
    \Pr(E \mid F) \;=\; \frac{\expc\big(\Pr(EF\mid \GG)\big)}{\Pr(F)}
    \;=\; \frac{\expc\big(\Pr(E \mid \GG) \ind(F)\big)}{\Pr(F)}.
\end{equation}
If~$\Pr(E \mid \GG) \ge p$ \as, then~\eqref{eq:cond-prob-sigma} yields~$\Pr(E \mid F) \ge p$.
If~$\Pr(E\mid F)\ge p$ for all~$F\in \GG$ with~$\Pr(F) >0$, then the
event~$\hat{F} = \{\Pr(E \mid \GG) < p\} \in \GG$ must have probability zero,
or else~\eqref{eq:cond-prob-sigma} implies~$\Pr(E\mid \hat{F}) < p$.
\end{proof}

Lemma~\ref{lem:equiv} elaborates on the equivalence among several probability inequalities. It is
useful for interpreting the conditions in Definitions~\ref{def:prob-ascent}
and~\ref{def:prob-ascent-nonsmooth}.

\begin{lemma}
    \label{lem:equiv}
    Let~$p\in[0,1]$ be a constant, $E$ and~$F$ be events, and~$\GG$ be a $\sigma$-algebra with~$E\in\GG$.
    Then the following three inequalities are equivalent to each other:
    \begin{equation*}
        \Pr(F \mid \GG) \;\ge\; p \ind(E^\comp), \qquad
        \Pr(F \cap E^\comp \mid \GG) \;\ge\; p \ind(E^\comp), \qquad
        \text{and} \qquad
        \Pr(E \cup F \mid \GG) \;\ge\; p.
    \end{equation*}
    In particular, if~$E\subseteq F$, then they are all equivalent to~$\Pr(F \mid \GG) \ge p$.
\end{lemma}
\begin{proof}
    We refer to the three inequalities as (a), (b), and (c), in left-to-right order.
    \\(a)~$\Rightarrow$~(b):~Since~$E^\comp \in \GG$, Lemma~\ref{lem:pef}
    yields~$\Pr(F \cap E^\comp \mid \GG) = \Pr(F \mid \GG) \ind(E^\comp) \;\ge\; p \ind(E^\comp)$.
    \\(b)~$\Rightarrow$~(c):~Since~$E \cup F = E \cup (F \cap E^\comp)$ and~$E \in \GG$, we have
    \[
        \Pr(E\cup F \mid \GG) \;=\; \Pr(E\mid \GG) + \Pr(F \cap E^\comp \mid \GG)
        \;=\; \ind(E) + \Pr(F\cap E^\comp \mid \GG)
        \;\ge\;\ind(E)+ p \ind(E^\comp)\;\ge\; p.
    \]
    (c)~$\Rightarrow$~(a):~Since~$(E \cup F)\cap E^\comp = F\cap E^\comp$
    and~$E^\comp \in \GG$, Lemma~\ref{lem:pef} leads to
    \[
        \Pr(F \mid \GG) \;\ge\; \Pr((E \cup F)\cap E^\comp \mid \GG) \;=\; \Pr(E \cup F \mid \GG)
        \ind(E^\comp) \;\ge\; p \ind(E^\comp).
    \]
    (c) reduces to~$\Pr(F\mid \GG) \ge p$ when~$E\subseteq F$.
\end{proof}

Lemma~\ref{lem:calculate_expectation} and the subsequent Corollary~\ref{cor:calculate_probability}
are helpful in the proofs of Propositions~\ref{prop:prob-asc} and~\ref{prop:prob-asc_nonsmooth}.
\begin{lemma}
  \label{lem:calculate_expectation}
  Let~$X$ and~$Y$ be random vectors. Consider a Borel measurable function~$h$
  such that $\expc(|h(X,Y)|)<\infty$ and define~$H(y) =\expc(h(X,y))$.
  \begin{enumerate}
    \item If~$X$ is independent of~$Y$, then~$\expc(h(X,Y)) =\expc(H(Y))$. \label{it:lemma-calculate-expectation-a}
    \item If~$X$ is independent of~$Y$~and a~$\sigma$-algebra~$\GG$, then~$\expc(h(X,Y) \mid \GG) = \expc(H(Y) \mid \GG)$. \label{it:lemma-calculate-expectation-b}
  \end{enumerate}
\end{lemma}
\begin{proof}
  Item~\ref{it:lemma-calculate-expectation-a} is a generalization of~\cite[Theorem~2.1.12]
  {Durrett_2019}, which considers random variables rather than random vectors. We omit its proof,
  which is a straightforward extension of that for~\cite[Theorem~2.1.12]{Durrett_2019}.

  Now we prove item~\ref{it:lemma-calculate-expectation-b} based on~\ref{it:lemma-calculate-expectation-a}. Since~$\expc(H(Y) \mid \GG)$ is $\GG$-measurable,
  the definition of conditional expectation tells us that we only need to verify
  \begin{equation}
      \label{eq:verify_cond_exp}
    \expc (h(X,Y)\ind(E)) \;=\; \expc (\expc(H(Y) \mid \GG) \ind(E))
  \end{equation}
  for all~$E\in\GG$.
  The right-hand
  side of~\eqref{eq:verify_cond_exp} equals~$\expc (H(Y) \ind(E))$
  due to the fact that~$E\in\GG$ and the tower property of conditional expectation. Hence, we only need to check
  \begin{equation}
      \label{eq:verify_cond_exp_equiv}
    \expc (h(X,Y)\ind(E)) \;=\; \expc (H(Y) \ind(E)).
  \end{equation}
  Denote~$\ind(E)$ by~$Z$ and define~$\hat{Y} = (Y,Z)$.
  Then~$X$ is independent of~$\hat{Y}$ by our assumption. Define
  \begin{equation*}
    \hat{h}(x,\hat{y}) = h(x,y)z \quad \text{and} \quad \hat{H}(\hat{y}) = \expc(\hat{h}(X,\hat{y})),
  \end{equation*}
  where~$\hat{y} = (y,z)$, with~$y$ and~$z$ having the same dimensions as $Y$ and~$Z$, respectively.
  Then we can apply item~\ref{it:lemma-calculate-expectation-a} to~$\hat{h}$ and~$\hat{H}$,
  obtaining
  \begin{equation}
      \label{eq:verify_exp_new_phi}
    \expc(\hat{h}(X,\hat{Y})) \;=\; \expc(\hat{H}(\hat{Y})).
  \end{equation}
  In addition, by the definition of~$\hat{H}$ and~$H$, we have
  \begin{equation}
      \label{eq:tld_H}
      \hat{H}(\hat{y}) \;=\; \expc(h(X,y)z) \;=\; H(y) z.
  \end{equation}
  Plugging~\eqref{eq:tld_H} and the definitions of~$\hat{h}$ into~\eqref{eq:verify_exp_new_phi},
  we obtain~$\expc(h(X,Y)Z) = \expc(H(Y)Z)$, which is~\eqref{eq:verify_cond_exp_equiv}.
  This completes the proof.
\end{proof}

\begin{remark}
  Taking expectation on both sides of the equality in item~\ref{it:lemma-calculate-expectation-b} of
  Lemma~\ref{lem:calculate_expectation}, we can recover item~\ref{it:lemma-calculate-expectation-a}
  by the tower property of conditional expectation.
  We also note that item~\ref{it:lemma-calculate-expectation-b} is a generalization
  of~\cite[Example~4.1.7]{Durrett_2019}~\textnormal{(}see also~\cite[page~148]{Cinlar_2011}\textnormal{)}, where~$\GG = \sigma(Y)$.
\end{remark}

    Recalling that the probability of an event is the expectation of its indicator function,
    we obtain Corollary~\ref{cor:calculate_probability}
    from item~\ref{it:lemma-calculate-expectation-b} of Lemma~\ref{lem:calculate_expectation}.
\begin{corollary}
    \label{cor:calculate_probability}
    Let~$X$ and~$Y$ be random vectors,~$\GG$ be a~$\sigma$-algebra, and~$h$ be a Borel measurable
    function. If~$X$ is independent of~$Y$~and~$\GG$, then
    $\Pr(h(X,Y) \ge 0\mid \GG) = \expc(P(Y) \mid \GG)$
    with~$P(y) = \Pr(h(X, y) \ge 0)$.
\end{corollary}

Now we present Lemma~\ref{lem:multi_rule}, which will help us prove Lemma~\ref{lem:shift}.
As a preparation,
given an event~$F$ with~$\Pr(F) >0$, we
let~$\Pr_F$ be the probability measure defined by~$\Pr_F(E) = \Pr(E\mid F)$ for every event~$E$,
and~$\Pr_F(\,\cdot \mid \GG)$ be the corresponding conditional probability with respect to a~$\sigma$-algebra~$\GG$.
Moreover, we let~$\expc_F$ denote the expectation under~$\Pr_F$, and~$\expc_F(\,\cdot \mid \GG)$
denote the corresponding conditional expectation. It is well known that
\begin{equation}
    \label{eq:multi_rule}
    \expc(X\ind(F)) \;=\; \expc_F(X) \Pr(F)
\end{equation}
for any random variable~$X$ (see, \eg,~\cite[Section~8.1]{Klenke_2020}). Consequently, we have
\begin{equation}
\label{eq:total_expect}
\expc(X \ind(F)) \;=\; \expc\left( \expc_F(X\mid\GG) \ind(F)\right).
\end{equation}
To see~\eqref{eq:total_expect}, multiply both sides of the
equality~$\expc_F(X) = \expc_F(\expc_F(X\mid \GG))$ by~$\Pr(F)$ and then apply~\eqref{eq:multi_rule}.

\begin{lemma}
\label{lem:multi_rule}
Let~$X$ be a random variable, $F$ be an event with~$\Pr(F)>0$, and~$\GG$ be a~$\sigma$-algebra.
  \begin{enumerate}
      \item \label{it:multi_rule}
          It holds that~$\expc(X \ind(F) \mid \GG) =  \expc_F(X\mid \GG) \,\Pr(F\mid \GG)$.
      \item \label{it:multi_rule_equiv}
          For any~$p\in[0, 1]$, we have the following equivalence:
          \begin{equation*}
          \expc(X \ind(F)\mid \GG) \;\ge\; p \Pr(F\mid \GG) \quad (\Pr\text{-\as})
          \qquad \Longleftrightarrow \qquad
          \expc_F(X\mid \GG) \;\ge\; p \quad (\Pr_F\text{-\as}).
      \end{equation*}
  \end{enumerate}
\end{lemma}
\begin{proof}
    \ref{it:multi_rule}~Since~$\expc_F( X \mid \GG ) \Pr(F\mid \GG)$ is~$\GG$-measurable,
    the definition of conditional expectation tells us that we only need to verify
    \begin{equation}\label{eq:multi-rule-verify}
        \expc( \expc_F(X\mid \GG) \, \Pr(F\mid \GG) \, \ind(E))
        \;=\;
        \expc(  X\ind(F)\, \ind(E))
    \end{equation}
    for all~$E\in \GG$.
    Denote~$Y = \expc_F(X\mid \GG)\ind(E)$. Then the left-hand side of~\eqref{eq:multi-rule-verify} equals
    \begin{equation*}
        \expc(Y\, \Pr(F\mid \GG))
        \;=\; \expc(Y\, \expc(\ind(F)\mid \GG))
        \;=\; \expc(\expc(Y \ind(F)\mid \GG) )
        \;=\; \expc(Y \ind(F) ).
    \end{equation*}
    To calculate~$\expc(Y \ind(F))$, we first note that~$Y = \expc_F(X \ind(E)\mid \GG)$ and then apply~\eqref{eq:total_expect}
    to the random variable~$X \ind(E)$, obtaining
    \begin{equation*}
        \expc(Y \ind(F) )
        \;=\; \expc\left(\expc_F(X\ind(E)\mid \GG) \, \ind(F)\right)
        \;=\; \expc\left([X \ind(E)]\ind(F)\right).
    \end{equation*}
    Therefore, equality~\eqref{eq:multi-rule-verify} holds.

    \ref{it:multi_rule_equiv}~Denote~$Z = \expc_F(X \mid \GG)$. According to~\ref{it:multi_rule}, we
    only need to prove the equivalence
    \begin{equation}\label{eq:lem-multi-rule-pf1}
        Z \Pr(F \mid \GG) \;\ge\; p \Pr(F \mid \GG) \quad (\Pr\text{-\as})
        \qquad \Longleftrightarrow \qquad
        Z \;\ge\; p \quad (\Pr_F\text{-\as}).
    \end{equation}
    To this end, defining the nonnegative random variable~$W = \ind(Z < p)\,\Pr(F \mid \GG)$, we observe that
    \begin{equation}
        \label{eq:lem_multi_rule_obs1}
            \big\{Z \Pr(F \mid \GG)  \,\ge\, p \Pr(F \mid \GG)\big\}
        \;=\; \big\{Z \ge p \,\text{ or }\, \Pr(F \mid \GG) = 0\}
        \;=\;  \big\{W = 0\},
    \end{equation}
    and that
    \begin{equation}
        \label{eq:lem_multi_rule_obs2}
        \Pr_F\!\left(Z < p \right)\Pr(F) \;=\; \Pr\left( \{Z < p\} \cap F \right)
        \;=\;
        \expc\big(\ind(Z < p) \Pr(F\mid \GG)\big) \;=\;  \expc(W).
    \end{equation}
    Therefore, we have the following two equivalences:
    \begin{equation*}
        Z \Pr(F \mid \GG) \,\ge\, p \Pr(F \mid \GG)
        \quad (\Pr\text{-\as})
        \qquad \Longleftrightarrow \qquad
        W  \,=\, 0
        \quad (\Pr\text{-\as})
        \qquad \Longleftrightarrow \qquad
        \Pr_F(Z < p) \,=\, 0,
    \end{equation*}
    where the first one is due to~\eqref{eq:lem_multi_rule_obs1}, while the second comes
    from~\eqref{eq:lem_multi_rule_obs2} and the fact that~$\Pr(F)>0$.
    Hence,~\eqref{eq:lem-multi-rule-pf1} holds. The proof is complete.
\end{proof}

\begin{remark}
    \label{rem:multi_rule}
    Item~\ref{it:multi_rule} of Lemma~\ref{lem:multi_rule} generalizes equality~\eqref{eq:multi_rule}.
    In light of~\ref{it:multi_rule}, item~\ref{it:multi_rule_equiv} shows that we can cancel~$\Pr(F\mid \GG)$ from both sides of
    the almost sure inequality~$\expc(X \ind(F) \mid \GG) \ge p \Pr(F \mid \GG)$,
    switching from~$\Pr$ to~$\Pr_F$ for the almost-sureness.
    For an event~$E$, item~\ref{it:multi_rule_equiv} with~$X = \ind(E)$ leads to the equivalence
    \begin{equation}
        \label{eq:multi_rule_equiv_event}
        \Pr(EF\mid \GG) \;\ge\; p \Pr(F\mid \GG) \quad (\Pr\text{-\as})
          \qquad \Longleftrightarrow \qquad
        \Pr_F(E\mid \GG) \;\ge\; p \quad (\Pr_F\text{-\as}).
    \end{equation}
    Note that~$\Pr(F \mid \GG)=\ind(F)$ if~$F \in \GG$, as is the case in
    the proof of Lemma~\ref{lem:shift}.
\end{remark}

Lemma~\ref{lem:inf_ineq} presents an elementary inequality needed in the proof of
Lemma~\ref{lem:strong_prob_exp_des}.
\begin{lemma}\label{lem:inf_ineq}
  Suppose that $k>k_0\ge 0$ and $0<q < p\le 1$. Then
  \begin{equation*}
    \inf_{t>0}~t(kq - k_0) + p (k-k_0)(\me^{-t} -1) \;\le\; -\frac{(q-p)^2}{2p}(k+k_0).
  \end{equation*}
\end{lemma}
\begin{proof}
  Considering~$t = \log (p/q)$, we only need to prove
    \begin{equation}\label{eq:pq_ineq}
        (kq-k_0)\log (p/q) + (k - k_0)(q - p) \;\le\; -\frac{(q-p)^2}{2p}(k+k_0).
    \end{equation}
    Regard the left-hand side of~\eqref{eq:pq_ineq} as a function of~$q$ and denote it by~$\varphi(q)$. Then
    \begin{equation*}
        \varphi(p) = 0,
        \quad
        \varphi'(p) \;=\; \frac{k_0}{p} - k_0 \;\ge\;0,
        \quad \text{and} \quad
        \varphi''(q) \;=\; -\frac{k}{q} - \frac{k_0}{q^2}.
    \end{equation*}
    By the Taylor expansion of $\varphi(q)$ at the point $p$, there exists a~$\xi\in(q,\,p)$ such that
    \begin{equation*}
      \varphi(q) \;=\; \varphi'(p)(q - p) + \frac{1}{2}\varphi''(\xi) (q-p)^2
      \;\le\; -\frac{(q-p)^2}{2}\left( \frac{k}{\xi} + \frac{k_0}{\xi^2} \right)
      \;\le\; -\frac{(q-p)^2}{2p}(k+k_0).
      \qedhere
    \end{equation*}
\end{proof}

Lemma~\ref{lem:optimality_measure} shows that~$\gap(0,\clarke f(\cdot))$
is lower semicontinuous if~$f$ is convex as mentioned in Remark~\ref{rem:optimality_measure}.

\begin{lemma}\label{lem:optimality_measure}
 Let~$f\mathrel{:}\RR^n \to \RR$ be convex. Then~$\mu(x) = \gap(0, \clarke f(x))$ is lower
 semicontinuous for~$x\in\RR^n$.
\end{lemma}

\begin{proof}
  Fix an~$x\in\RR^n$ and an~$\varepsilon>0$. By~\cite[Corollary~24.5.1]{Rockafellar_1970}, there exists a~$\delta>0$ such that
  \begin{equation*}
    \clarke f(y) \subseteq \clarke f(x) + \BB(0,\varepsilon) \quad \text{for all~$y\in\BB(x,\delta)$}.
  \end{equation*}
  This implies that
  \begin{equation*}
      \gap(0, \clarke f(y)) \;\ge\; \gap(0, \clarke f(x)) - \varepsilon \quad \text{for all~$y\in\BB(x,\delta)$}.
  \end{equation*}
  Hence,~$\mu$ is lower semicontinuous.
\end{proof}

\section{Discussions about the definition of probabilistic ascent}\label{app:disc_def}

This section discusses Definition~\ref{def:prob-ascent} of probabilistic ascent, especially the role
of the indicator~$\ind(G_k \neq 0)$. As Remark~\ref{rem:cm_not_affect} will clarify,
this indicator ensures that the concept of
probabilistic ascent and the subsequent theory are invariant to the value
of~$\cm(\,\cdot\,, 0)$, which is purposefully unspecified in Definition~\ref{def:cm}.
Removing this indicator from Definition~\ref{def:prob-ascent} would make the theory rely
on~$\cm(\,\cdot\,, 0)$, leading to an undesirable restriction if one
defines~$\cm(\,\cdot\,, 0) =1$ following~\cite{Gratton_Royer_Vicente_Zhang_2015}, as will be
explained in Remark~\ref{rem:wrong-def}.

We begin with Proposition~\ref{prop:prob-asc-def-equiv}, which provides two equivalent reformulations
for the inequality in condition~\eqref{eq:prob-asc-def} of Definition~\ref{def:prob-ascent}.
This proposition can be proved by applying Lemma~\ref{lem:equiv} to the events~$E=\{G_k=0\}$
and~$F=\{\cm(\DD_k, -G_k) \le 0\}$ while noting
that~$E\cup F = \{\min_{\ddd\in\DD_k}\!\ddd^\trs G_k \,\ge\, 0\}$.

\begin{proposition}
    \label{prop:prob-asc-def-equiv}
    Let $p \in [0,1]$. Consider Algorithm~\ref{alg:direct_search_r} with~$f$ being continuously
    differentiable on $\RR^n$.
    For each~$k\ge 0$, the following inequalities are equivalent to each other.
    \begin{enumerate}
        \item $\Pr\left(\cm(\DD_k, -G_k) \le 0 \mid \FF_{k-1}) \ge p \ind(G_k \ne 0\right)$. \label{it:prob-asc-def-equiv-a}
        \item $\Pr\left(\{\cm(\DD_k, -G_k) \le 0\} \cap \{G_k\ne 0\} \mid \FF_{k-1}\right) \ge p \ind(G_k \ne 0)$. \label{it:prob-asc-def-equiv-b}
        \item $\Pr\left(\min_{\ddd\in\DD_k} \ddd^\trs G_k \ge 0 \mid \FF_{k-1}\right) \ge p$. \label{it:prob-asc-def-equiv-c}
    \end{enumerate}
\end{proposition}

\begin{remark}\label{rem:cm_not_affect}
Neither item~\ref{it:prob-asc-def-equiv-b} nor~\ref{it:prob-asc-def-equiv-c} in
Proposition~\ref{prop:prob-asc-def-equiv} relies on the value of~$\cm(\,\cdot\,, 0)$.
Therefore, condition~\eqref{eq:prob-asc-def} based on~\ref{it:prob-asc-def-equiv-a}
is independent of~$\cm(\,\cdot\,, 0)$.
Consequently, no matter how we define~$\cm(\,\cdot\,, 0)$, Definition~\ref{def:prob-ascent} of
probabilistic ascent is invariant, and the results in this paper hold without any modification.
\end{remark}%

\begin{remark}
\label{rem:prob_asc_alt}
Since items~\ref{it:prob-asc-def-equiv-a} and~\ref{it:prob-asc-def-equiv-c}
in Proposition~\ref{prop:prob-asc-def-equiv} are equivalent, condition~\eqref{eq:prob-asc-def} is equivalent to
\begin{equation}
    \label{eq:prob-asc-def-equiv-c}
       \Pr\Big(\min_{\ddd\in\DD_k} \ddd^\trs G_k \ge 0 \Mid \FF_{k-1} \Big) \;\ge\; p \quad
       \text{for each~$k\ge 0$}.
\end{equation}
This is exactly condition~\eqref{eq:prob-asc-def-nonsmooth} in the continuously differentiable case,
where we have~$f^\circ(X_k;\ddd) = \ddd^\trs G_k$.
\end{remark}

What if we dropped the indicator~$\ind(G_k \ne 0)$
from the definition of probabilistic ascent, adopting an alternative definition that requires
  \begin{equation}\label{eq:wrong-def}
    \Pr\left(\cm\left(\DD_k,-G_k\right) \le 0 \mid \FF_{k-1} \right) \;\ge\; p\quad \text{for each~$k\ge 0$}
  \end{equation}
in place of condition~\eqref{eq:prob-asc-def}?
Above all, every result requiring~$p$-probabilistic ascent in this paper would still hold,
since~\eqref{eq:wrong-def} is stronger than~\eqref{eq:prob-asc-def}.
However, as we will explain in Remark~\ref{rem:wrong-def},
the theory built on this alternative definition would rely on the value of~$\cm(\,\cdot\,,
0)$, which would imply an undesirable limitation if one defines~$\cm(\,\cdot\,, 0) =1$ as
in~\cite{Gratton_Royer_Vicente_Zhang_2015}.
\begin{remark}\label{rem:wrong-def}
  Suppose that we define~$\cm(\,\cdot\,, 0) =1$ following~\cite{Gratton_Royer_Vicente_Zhang_2015}.
  Then condition~\eqref{eq:wrong-def} cannot be satisfied for any~$p>0$ unless
  \begin{equation}
      \label{eq:pgk0}
      \Pr(G_k = 0) \;=\; 0 \quad \text{for each}~k\ge 0.
  \end{equation}
  Condition~\eqref{eq:pgk0} means that Algorithm~\ref{alg:direct_search_r} almost never steps on a stationary point, which is a restriction that we do not want to impose.
  To see why~\eqref{eq:wrong-def} with~$p>0$ necessitates~\eqref{eq:pgk0},
  let us assume that~$\Pr(G_k = 0) > 0$ for a certain~$k\ge 0$. Then~\eqref{eq:wrong-def} and Lemma~\ref{lem:prob_sigma_to_event} will lead
  to the contradiction that
  \begin{equation*}
      p \;\le\;
      \Pr(\cm\left(\DD_k, -G_k \right) \le 0 \mid G_k = 0) \;=\;
      \Pr(\cm\left(\DD_k, 0 \right) \le 0 \mid G_k = 0) \;=\;
      0,
  \end{equation*}
  where the last step is because~$\cm(\DD_k, 0)=1$.
\end{remark}

To illustrate Remark~\ref{rem:wrong-def}, let us recall Example~\ref{exp:simple_illustrate},
choosing~$x_0 = \alpha_0$ in particular.
Then we have
\[
    \Pr(G_1 = 0) = \Pr(\ddd_0 = -1) = \frac{1}{2},
\]
violating~\eqref{eq:pgk0} for~$k=1$. Consequently, Remark~\ref{rem:wrong-def} shows that
condition~\eqref{eq:wrong-def} with~$p>0$ cannot
hold if one defines~$\cm(\,\cdot\,, 0) = 1$ as in~\cite{Gratton_Royer_Vicente_Zhang_2015},
making the theory based on~\eqref{eq:wrong-def} inapplicable to such a simple example,
even though~$\{\DD_k\}$ fits our definition of~$1/2$-probabilistic ascent according to
Proposition~\ref{prop:prob-asc}.
This example also clarifies why~\eqref{eq:pgk0} is undesirable to impose
when analyzing randomized algorithms like Algorithm~\ref{alg:direct_search_r}, although it is not
uncommon to assume that algorithms never step on stationary points in the deterministic case~(\eg,~\cite[Section~1]{Powell_1975}).

Before ending this section, we propose Definition~\ref{def:prob-des-alt} as an alternative to
Definition~\ref{def:p-k_descent} for probabilistic descent.
In this definition,~$\ind(G_k \ne 0)$ plays a role like in Definition~\ref{def:prob-ascent}.

\begin{definition}[Alternative definition of probabilistic descent]\label{def:prob-des-alt}
  Identical to Definition~\ref{def:p-k_descent} except that we replace
  condition~\eqref{eq:p-k_descent-def} with
  \begin{equation}
    \label{eq:prob-des-def-alt}
    \Pr\left(\cm\left(\DD_k,-G_k\right) \ge \kappa \mid \FF_{k-1}\right) \;\ge\; p\ind(G_k \ne 0) \quad \text{for each~$k\ge 0$}.
  \end{equation}
\end{definition}

Similar to Definition~\ref{def:prob-ascent}, Definition~\ref{def:prob-des-alt}
is invariant to the value of~$\cm(\,\cdot\,, 0)$.
If~$\cm(\,\cdot\,, 0) = 1$ as in~\cite{Gratton_Royer_Vicente_Zhang_2015},
it is equivalent to Definition~\ref{def:p-k_descent},
because we have~$\{G_k = 0\} \subseteq \left\{\cm(\DD_k, -G_k) \ge \kappa \right\}$ in this case,
which ensures the equivalence by Lemma~\ref{lem:equiv}.
However, if one chooses to define~$\cm(\,\cdot\,, 0) < \kappa$ (\eg,~$\cm(\,\cdot\,, 0) = 0$ may be appealing for
symmetry),
the argument in Remark~\ref{rem:wrong-def} can be adapted to show that
Definition~\ref{def:p-k_descent} also necessitates~\eqref{eq:pgk0} when~$p>0$, whereas
Definition~\ref{def:prob-des-alt} is free from this restriction.%

\section{Proofs of Propositions~\ref{prop:prob-asc},~\ref{prop:critical},~\ref{prop:prob-asc_nonsmooth},
and Lemmas~\ref{lem:strong_prob_exp_des}--\ref{lem:shift}}
\label{app:proofs}

This section proves several propositions and lemmas that appeared in the main text. They will be
proved in the order of their appearance.

Proposition~\ref{prop:prob-asc} can be established with the help of Corollary~\ref{cor:calculate_probability}.

\begin{proof}[Proof of Proposition~\ref{prop:prob-asc}]
    Fix an arbitrary~$k\ge 0$.
    Due to Proposition~\ref{prop:prob-asc-def-equiv}, it suffices to prove that
  \begin{equation}
    \label{eq:verify_prob-asc}
    \Pr\left(\min_{\ddd\in\DD_k} \ddd^\trs G_k \ge 0 \MID \FF_{k-1} \right) \;\ge\; 2^{-m}.
  \end{equation}
  By our assumption,~$\DD_k$ is independent of~$G_k$ and~$\FF_{k-1}$. Hence,
  Corollary~\ref{cor:calculate_probability} implies that the conditional probability
  in~\eqref{eq:verify_prob-asc} equals~$\expc\left(P(G_k) \mid \FF_{k-1}\right)$ with
\begin{equation}
    \label{eq:Pg}
    P(g) \;=\; \Pr\left(\min_{\ddd\in\DD_k} \ddd^\trs g \ge 0\right)
    \;=\; \Pr\left(\ddd^\trs g \ge 0 \text{~for each~} \ddd\in\DD_k \right).
\end{equation}
Since~$\DD_k$ consists of~$m$ independent random vectors uniformly distributed on the unit sphere,
the right-hand side of~\eqref{eq:Pg} is at least~$2^{-m}$ (it is exactly~$2^{-m}$ if~$g\ne 0$
and is~$1$ if~$g=0$). Therefore,~\eqref{eq:verify_prob-asc} holds.
\end{proof}

Indeed,~$\expc(P(G_k)\mid \FF_{k-1})$ in the above proof equals~$P(G_k)$ since~$G_k$
is~$\FF_{k-1}$-measurable, but we do not need this fact to establish the desired inequality.

Lemma~\ref{lem:strong_prob_exp_des} can be obtained using the moment method for
deriving Chernoff bounds~\cite{Mulzer_2018}.
\begin{proof}[Proof of Lemma~\ref{lem:strong_prob_exp_des}]
  The inequality in~\eqref{eq:chernoff_bound} holds trivially when~$k \le k_0$,
  because~\eqref{eq:YUE} implies that~$E_{k_0}\subseteq\{\overline{Y}_k = 0\}$ in this case,
  rendering~$\Pr(\overline{Y}_k\ge 1-q\mid E_{k_0}) = 0$ since~$q<1$.
  Let us focus on the nontrivial case where~$k > k_0\ge 0$.

  Fixing an arbitrary~$t>0$, we first make two claims: one is
  \begin{equation}\label{eq:first-claim}
    \Pr \left( 1 - \overline{Y}_k\le q \mid E_{k_0} \right)
    \;\le\;
    \me^{t (k q - k_0)}\;\expc\left( \prod_{\ell=k_0}^{k-1}\me^{-t(1-Y_\ell)} \MID E_{k_0}\right),
  \end{equation}
  and the other is
  \begin{equation}\label{eq:second-claim}
    \expc\left( \prod_{\ell=k_0}^{k-1}\me^{-t(1-Y_\ell)} \MID E_{k_0}\right) \;\le\; \exp[p (k-k_0)(\me^{-t} -1)].
  \end{equation}
  Once inequalities~\eqref{eq:first-claim} and~\eqref{eq:second-claim} are proved, we will have
  \begin{equation*}
    \Pr \left( 1 - \overline{Y}_k \le q \mid E_{k_0} \right) \;\le\; \exp[t(kq - k_0) + p (k-k_0) (\me^{-t} -1)],
  \end{equation*}
  which will render~\eqref{eq:chernoff_bound} according to Lemma~\ref{lem:inf_ineq}. We now prove the two claims by standard techniques.

  For~\eqref{eq:first-claim}, by the definitions of~$\overline{Y}_k$
  and~$E_{k_0}$ in~\eqref{eq:YUE} as well as Markov's inequality, we have
  \begin{equation*}%
    \begin{aligned}
      \Pr \left( 1-\overline{Y}_k\le q \mid E_{k_0} \right)
      & \;=\; \Pr \left( \exp\Big[-t\sum_{\ell=0}^{k-1}(1-Y_\ell)\Big]\ge \me^{-t k q} \MID E_{k_0}\right)\\
      & \;\le\; \me^{t k q}~ \expc\left( \prod_{\ell=0}^{k-1}\me^{-t(1-Y_\ell)} \MID E_{k_0}\right)
       \;=\; \me^{t (k q - k_0)}~\expc\left( \prod_{\ell=k_0}^{k-1}\me^{-t(1-Y_\ell)} \MID E_{k_0}\right),
    \end{aligned}
  \end{equation*}
  where the last equality is because~$\prod_{\ell=0}^{k_0-1}\me^{-t(1-Y_\ell)} = \me^{-tk_0}$ when~$E_{k_0}$ happens.

  For~\eqref{eq:second-claim}, we use the tower property of conditional expectation to get
  \begin{equation}\label{eq:tower}
    \expc\left( \prod_{\ell=k_0}^{k-1}\me^{-t(1-Y_\ell)} \MID \FF_{k_0-1}\right) \;=\; \expc\left( \expc\left( \me^{-t(1-Y_{k-1})} \mid \FF_{k-2} \right)\prod_{\ell=k_0}^{k-2}\me^{-t(1-Y_\ell)} \MID \FF_{k_0-1}\right),
  \end{equation}
  with~$\prod_{\ell=k_0}^{k-2}\me^{-t(1-Y_\ell)} = 1$ when $k=k_0+1$.
  By condition~\eqref{eq:Y_k-condition}, we have
  \begin{equation}\label{eq:condition_expect}
    \expc\left( \me^{-t(1-Y_{k-1})} \Mid \FF_{k-2} \right)
    \;\le\; p \me^{-t} + (1 - p) \;\le\; \exp(p \me^{-t} -p),
  \end{equation}
  where the last inequality is because~$x+1\le \me^x$ for all~$x$.
  By equality~\eqref{eq:tower} and inequality~\eqref{eq:condition_expect}, we have
  \begin{equation}\label{eq:recursive}
    \begin{aligned}
      \expc\left( \prod_{\ell=k_0}^{k-1}\me^{-t(1-Y_\ell)} \MID \FF_{k_0-1}\right)
      & \;\le\; \exp[p (\me^{-t} -1)]~\expc\left( \prod_{\ell=k_0}^{k-2}\me^{-t(1-Y_\ell)} \MID \FF_{k_0-1}\right)\\
      & \;\le\; \exp[p (k-k_0)(\me^{-t} -1)],
    \end{aligned}
  \end{equation}
  the second inequality following from the recursive application of the first one.
  Since~$E_{k_0}\in \FF_{k_0-1}$
  and~$\Pr(E_{k_0}) >0$ by Remark~\ref{rem:Pr_E_k}, inequality~\eqref{eq:recursive} implies~\eqref{eq:second-claim} by Lemma~\ref{lem:prob_sigma_to_event}.
  The proof is complete.
\end{proof}

Lemma~\ref{lem:shift} is a straightforward consequence of Lemma~\ref{lem:multi_rule}, or, more precisely,
Remark~\ref{rem:multi_rule}.

\begin{proof}[Proof of Lemma~\ref{lem:shift}]
    Since~$p>0$, the probability measure~$\Pr(\,\cdot\mid E_{k_0})$ is well defined according to
    Remark~\ref{rem:Pr_E_k}.
    Fix an integer~$k\ge 0$. Then~$E_{k_0}\in\FF_{k_0-1}\subseteq \FF_{k_0+k-1}=\tilde{\FF}_{k-1}$
    by the definitions of~$\{\FF_k\}$ and~$\{\tilde{\FF}_k\}$.
    Thus, condition~\eqref{eq:Y_k-condition} and Lemma~\ref{lem:pef} yield
    \[
        \Pr(\{\tilde{Y}_k = 0\} \cap E_{k_0} \mid \tilde{\FF}_{k-1})
        \;=\; \Pr(Y_{k_0+k} = 0 \mid \FF_{k_0+k -1})\, \ind(E_{k_0}) \;\ge\; p \ind(E_{k_0}).
    \]
    Hence, recalling that~$\tilde{\Pr}(\cdot)=\Pr(\,\cdot\mid E_{k_0})$,
    we have~$\tilde{\Pr}(\tilde{Y}_k=0\mid\tilde{\FF}_{k-1}) \ge p$ according to Remark~\ref{rem:multi_rule}.
\end{proof}

    Proposition~\ref{prop:critical} can be proved using
    L\'{e}vy's Conditional Borel--Cantelli Lemma~\cite[Corollaire~68]{Levy_1937}~(see
    also~\cite[Theorem~4.3.4]{Durrett_2019} and~\cite[Exercise~11.2.7]{Klenke_2020}).
\begin{proof}[Proof of Proposition~\ref{prop:critical}]
    Since~$Y_kU_k\ge 0$ for each~$k\ge 0$, it suffices to prove that
    \begin{equation}
        \label{eq:nto0}
        \Pr\left( \limsup_{k\to \infty}Y_kU_k \,>\, 0 \right) \;=\; 1.
    \end{equation}
    To this end, we first note that
    \begin{equation}
        \label{eq:unto0}
        \Pr\left( \limsup_{k\to \infty} U_k \,>\, 0 \right) \;=\; 1.
    \end{equation}
    Indeed, since~$\log U_k =\sum_{\ell=0}^{k-1} [Y_\ell \log \gamma + (1-Y_\ell)\log \theta]$,
    condition~\eqref{eq:critical_Y_k} and the definition of~$p_*$ in~\eqref{eq:define-p_*} ensure
    that~$\{\log U_k\}$ is a martingale with respect
    to~$\{\FF_{k-1}\}$, and this martingale has bounded increments. Hence,~\cite[Theorem~4.3.1]{Durrett_2019}
    indicates that~$\limsup_{k}\,(\log U_k) > -\infty$~\as, which is equivalent to~\eqref{eq:unto0}.

    To finish the proof of~\eqref{eq:nto0}, we will demonstrate that
    \begin{equation}
        \label{eq:io}
        \Pr\left(Y_kU_k \ge \varepsilon \;~\io\right)
        \;\ge\; \Pr\left(U_k \ge \varepsilon \;~\io\right)
        \quad \text{ for all } \varepsilon > 0.
    \end{equation}
    Once~\eqref{eq:io} is established, plugging~$\varepsilon = \ell^{-1}$ into it
    and then taking limit as~$\ell \to \infty$ will render
    $\Pr(\limsup_k Y_kU_k >0) \,\ge\, \Pr(\limsup_k U_k >0)$, which will lead to~\eqref{eq:nto0}
    in light of~\eqref{eq:unto0}.

    Now we prove~\eqref{eq:io} for an arbitrary~$\varepsilon > 0$. For each~$k\ge 0$,
    since~$\{Y_kU_k \ge \varepsilon\} = \{Y_k = 1\} \cap \big\{U_k \ge \varepsilon\big\}$
    and~$\left\{ U_k \ge \varepsilon\right\}\in\FF_{k-1}$, Lemma~\ref{lem:pef}
    and condition~\eqref{eq:critical_Y_k} yield
    \begin{equation}
        \label{eq:observe}
        \begin{split}
        \Pr(Y_kU_k \ge \varepsilon \mid \FF_{k-1})
        \;=\; \Pr(Y_k = 1 \mid \FF_{k-1}) \, \ind\!\left(U_k \ge \varepsilon\right)
        \;=\;  (1-p_*) \,\ind\!\left(U_k \ge \varepsilon\right).
        \end{split}
    \end{equation}
    Recalling that~$0<\theta<1\le \gamma$, we have~$p_* <1$ according to~\eqref{eq:define-p_*}.
    Hence,~\eqref{eq:observe} implies that
    \begin{equation}
        \label{eq:bcsum}
    \left\{\sum_{k=0}^\infty \Pr(Y_kU_k \ge \varepsilon \mid \FF_{k-1}) =\infty\right\}
    \;\supseteq\; \left\{\sum_{k=0}^\infty \ind(U_k \ge \varepsilon) = \infty\right\}
    \;=\; \left\{U_k \ge \varepsilon \,~\io\right\}.
    \end{equation}
    Meanwhile, the left-hand side of~\eqref{eq:bcsum} has the same probability as
    the event~$\left\{Y_kU_k \ge \varepsilon \,~\io\right\}$ by
    L\'{e}vy's Conditional Borel--Cantelli Lemma.
    Therefore,~\eqref{eq:io} holds and the proof is complete.
\end{proof}

The proof of Proposition~\ref{prop:prob-asc_nonsmooth} is similar to that of Proposition~\ref{prop:prob-asc}.

\begin{proof}[Proof of Proposition~\ref{prop:prob-asc_nonsmooth}]
  Similar to the proof of Proposition~\ref{prop:prob-asc}, Corollary~\ref{cor:calculate_probability}
  implies that the conditional probability in condition~\eqref{eq:prob-asc-def-nonsmooth}
  equals~$\expc\left(P(X_k) \mid \FF_{k-1}\right)$ with
\begin{equation}
    \label{eq:Px}
    P(x) \;=\; \Pr\left(\min_{\ddd\in\DD_k} f^{\circ}(x; \ddd)\ge 0\right).
\end{equation}
For any~$x \in\RR^n$, picking a~$g\in \clarke f(x)$, we
have~$P(x)\ge \Pr\left(\min_{\ddd\in\DD_k} \ddd^\trs g\ge 0\right)$, which is at least~$2^{-m}$
as in the proof of Proposition~\ref{prop:prob-asc}.
Hence,~\eqref{eq:prob-asc-def-nonsmooth} holds with~$p=2^{-m}$ and the proof is complete.
\end{proof}

\section{(Non-)Measurability of iterates with respect to polling directions}
\label{sec:measurability}

In this section, we discuss when the iterates of Algorithm~\ref{alg:direct_search_r} are
measurable with respect to the polling directions, and when they are not.
Often omitted in literature, this type of discussion is
essential for the mathematical rigour of our analysis. Indeed,
as we will see in Example~\ref{exp:measurable-polling}, the measurability can fail for some implementations of
Algorithm~\ref{alg:direct_search_r}.
For the concept of measurability, we refer to~\cite[Section~1.2]{Durrett_2019}.

Lemma~\ref{lem:measurable-polling} establishes the measurability of the iterates
for certain implementations of Algorithm~\ref{alg:direct_search_r},
covering~\cite[Algorithm~2.1]{Gratton_Royer_Vicente_Zhang_2015}.
The proof is elementary, but it clarifies the role of the polling strategy in the measurability.

\begin{lemma}\label{lem:measurable-polling}
  Let~$m$ be a positive integer and~$f$ be continuous on~$\RR^n$.
  Consider Algorithm~\ref{alg:direct_search_r} with the following configuration for each~$k\ge 0$.
  \begin{enumerate}
    \item Generate~$\DD_k = \{\ddd_k^1, \dots, \ddd_k^m\}$ with~$\ddd_k^1, \dots, \ddd_k^m$ being random vectors. \label{it:measurable-polling-a}
    \item Set the order of function evaluations as~$f(X_k+A_k\ddd_k^1),\dots,f(X_k+A_k\ddd_k^m)$ before polling. \label{it:measurable-polling-b}
    \item Use either opportunistic polling or complete polling. \label{it:measurable-polling-c}
  \end{enumerate}
  Let~$\FF_{k}^\DD = \sigma(\DD_0,\dots,\DD_{k})$ for each~$k\ge 0$ and~$\FF_{-1}^\DD = \{\emptyset,\Omega\}$.
  Then~$X_k$ is~$\FF_{k-1}^\DD$-measurable for each~$k\ge 0$.
\end{lemma}
\begin{proof}
  We will prove by induction that~$X_k$ and~$A_k$ are both~$\FF_{k-1}^\DD$-measurable for each~$k\ge 0$.
  The base case~$k=0$ holds trivially since~$X_0$ and~$A_0$ are not random.
  Assuming that~$X_{k}$ and~$A_k$ are~$\FF_{k-1}^\DD$-measurable, let us prove that~$X_{k+1}$
  and~$A_{k+1}$ are both~$\FF_{k}^\DD$-measurable.
  Before starting, note that the induction hypothesis implies
  that~$X_k$ and~$A_k$ are~$\FF_{k}^\DD$-measurable since~$\FF_{k-1}^\DD\subseteq\FF_{k}^\DD$.
  Define~$\ddd_k^0=0$ and
\begin{equation*}
    V^i \;=\; f(X_k + A_k \ddd_k^i), \quad i=0,1,\dots,m.
\end{equation*}
Then each~$V^i$ is~$\FF_{k}^\DD$-measurable since~$f$ is continuous.
$\rho(A_k)$ is also~$\FF_{k}^\DD$-measurable as~$\rho$ is monotone.

Now, we consider the case of complete polling. In this case,
  \begin{equation}
      \label{eq:X_kp_comp}
      X_{k+1} \;=\; X_k + A_k \sum_{i=1}^m \ddd_k^i W^i,
  \end{equation}
  where~$W^i$ ($i = 1, \dots, m$) is the indicator defined by
\begin{equation*}
    W^i \;=\; \ind\!\left( i \text{ is the smallest integer such that } V^i = \min\{V^1, \dots, V^m\},
    \text{ and } V^0 - V^i > \rho(A_k)\right).
\end{equation*}
Note that at most one of~$W^1,\dots,W^m$ is~$1$, and they are all~$0$ if complete polling
fails.
Moreover,
\begin{equation*}
    W^i \;=\;
        \Bigg[
            \prod_{j=1}^{i-1} \ind\!\left(V^i < V^j\right)
        \prod_{j=i+1}^m \ind\!\left(V^i \le V^j\right)
        \Bigg]
        \ind\!\left(V^0 - V^i > \rho(A_k)\right),
\end{equation*}
which is~$\FF_{k}^\DD$-measurable due to the~$\FF_{k}^\DD$-measurability of~$V^0,\dots,V^m$
and~$\rho(A_k)$. Therefore, $X_{k+1}$ is~$\FF_{k}^\DD$-measurable according
to~\eqref{eq:X_kp_comp}. Consequently, $A_{k+1}$ is~$\FF_{k}^\DD$-measurable by the recurrence
relation~\eqref{eq:Ak_recur} and the induction hypothesis.
The induction finishes for complete polling.

The case of opportunistic polling can be handled similarly. In this case,
equation~\eqref{eq:X_kp_comp} holds with
\begin{equation*}
    \begin{split}
    W^i
    & \;=\; \ind\!\left( i \text{ is the smallest integer such that } V^0 - V^i > \rho(A_k)\right) \\
    & \;=\;
    \Bigg[\prod_{j=1}^{i-1} \ind\!\left(V^0 - V^j \le \rho(A_k) \right)\Bigg]
    \ind\!\left(V^0 - V^i > \rho(A_k)\right),
    \end{split}
\end{equation*}
which is~$\FF_{k}^\DD$-measurable. Everything else is the same as complete polling.
\end{proof}

However, if the polling in Algorithm~\ref{alg:direct_search_r} involves randomness beyond the
polling directions, then~$X_{k}$ may not be~$\FF_{k-1}^\DD$-measurable.
This is illustrated by Example~\ref{exp:measurable-polling}.
For this reason, our analysis uses
$\FF_k=\sigma (\DD_0, X_1, \dots, \DD_k, X_{k+1})$ rather than~$\FF_k^\DD$ as the filtration.
\begin{example}
  \label{exp:measurable-polling}
  Let~$m$ be a positive integer and~$f$ be continuous on~$\RR^n$.
  Consider Algorithm~\ref{alg:direct_search_r} with the following configuration for each~$k\ge 0$.
  \begin{enumerate}
    \item Generate $\DD_k = \{\ddd_k^1, \dots, \ddd_k^m\}$ with~$\ddd_k^1, \dots, \ddd_k^m$ being random vectors. \label{it:measurable-polling-ex-a}
    \item Pick a random permutation~$\pi_k$ of~$\{1,\dots,m\}$. \label{it:measurable-polling-ex-b}
    \item Set the order of function evaluations as~$f(X_k+A_k\ddd_k^{\pi_k(1)}),\dots,f(X_k+A_k\ddd_k^{\pi_k(m)})$ before polling. \label{it:measurable-polling-ex-c}
    \item Use opportunistic polling. \label{it:measurable-polling-ex-d}
  \end{enumerate}
Since the calculation of~$X_k$ involves~$\pi_{k-1}$,
we cannot guarantee the~{$\FF_{k-1}^{\DD}$-measurability} of~$X_k$
if~$\pi_{k-1}$ is not~{$\FF_{k-1}^\DD$-measurable},
or informally, if~$\pi_{k-1}$ contains randomness beyond~$\FF_{k-1}^\DD$.
For example,
similar to~\cite[Section~4]{Custodio_Vicente_2007},
we can define~$\pi_{k-1}$
by ranking the directions in~$\DD_{k-1}$ according to a stochastic oracle
independent of~$\DD_{k-1}$.
Then~$X_k$ is measurable with respect to~$\sigma(\DD_0, \pi_0, \dots, \DD_{k-1}, \pi_{k-1})$,
but not necessarily with respect to~$\FF_{k-1}^\DD$.
\end{example}

\end{appendices}

\end{document}